\documentclass[reqno]{amsart}

\title[ ]{Harmonic problems arising from continuous time random walks limit processes}

\author{Ivan Bio\v{c}i\'{c}}
\author{Bruno Toaldo}
	\date{\today}
\thanks{The authors acknowledge financial support under the National Recovery and Resilience Plan (NRRP), Mission 4, Component 2, Investment 1.1, Call for tender No. 104 published on 2.2.2022 by the Italian Ministry of University and Research (MUR), funded by the European Union – NextGenerationEU– Project Title “Non–Markovian Dynamics and Non-local Equations” – 202277N5H9 - CUP: D53D23005670006 - Grant Assignment Decree No. 973 adopted on June 30, 2023, by the Italian Ministry of University and Research (MUR)}

\thanks{The author B. Toaldo would like to thank the Isaac Newton Institute for Mathematical Sciences, Cambridge, for support and hospitality during the programme Stochastic systems for anomalous diffusion, where work on this paper was undertaken. This work was supported by EPSRC grant EP/Z000580/1.}

\thanks{The author I. Biočić acknowledges financial support by the European Union – NextGenerationEU through the National Recovery and Resilience Plan 2021-2026 Institutional grant of University of Zagreb Faculty of Science (IK IA 1.1.3. Impact4Math), as well as the support by Croatian Science Foundation through the project IP-2025-02-8793.}

\usepackage{amssymb, amsmath, amsfonts}
\usepackage{natbib}
\usepackage{lscape,comment,upgreek}
\usepackage{color}
\usepackage{mathtools}
\usepackage{dsfont}
\usepackage{mathrsfs}
\usepackage{bm}
\usepackage{graphicx}
\usepackage{wrapfig}
\usepackage{blindtext}
\usepackage[left=3cm,right=3cm,top=3cm,bottom=3cm]{geometry}
\usepackage{caption}
\usepackage{enumitem}
\usepackage{tikz}
\usetikzlibrary{arrows,decorations.markings,arrows.meta}
\usetikzlibrary{calc}

\usepackage{cancel}

\usepackage{soul}

\usepackage[pdftex,plainpages=false,colorlinks,hyperindex,bookmarksopen,linkcolor=red,citecolor=blue,urlcolor=blue]{hyperref}

\numberwithin{equation}{section}

\DeclareMathAlphabet{\mathpzc}{OT1}{pzc}{m}{it}

\DeclareMathOperator\supp{supp}	
\newcommand{\R}{\mathbb{R}}
\newcommand{\N}{\mathbb{N}}
\newcommand{\pr}{\mathds{P}} 
\newcommand{\ex}{\mathds{E}} 

\newcommand{\1}{\mathds 1}

\newcommand{\II}{\mathds{I}}

\newcommand{\BB}{\mathcal{B}}

\newcommand{\DD}{\mathcal{D}}

\newcommand{\wt}{\widetilde}

\newcommand{\red}[1]{{\color{black}#1}}

\bibpunct{[}{]}{;}{n}{,}{,}
\usepackage{algorithm2e}
\usepackage{algpseudocode}
\RestyleAlgo{ruled}

\usepackage{bbm} 
\newtheorem{theorem}{Theorem}[section] 
\newtheorem{lemma}[theorem]{Lemma}
\newtheorem{proposition}[theorem]{Proposition}
\newtheorem{definition}[theorem]{Definition}
\theoremstyle{remark}
\newtheorem{remark}[theorem]{Remark}

\newtheorem{example}[theorem]{Example}
\numberwithin{equation}{section} 
\usepackage{authblk}

\usepackage{hyperref}
\hypersetup{
	colorlinks=true,
}
\usepackage{caption}
\usepackage{subcaption}
\usepackage{graphicx}
\usepackage[pagewise]{lineno}
\usepackage{float} 

\newcommand{\EE}{\mathds{E}}

\DeclareMathAlphabet{\mathpzc}{OT1}{pzc}{m}{it}

\allowdisplaybreaks

\newtheorem*{assumption*}{\assumptionnumber}
\providecommand{\assumptionnumber}{}
\makeatletter
\newenvironment{assumption}[2]
{%
	\renewcommand{\assumptionnumber}{(\textbf{#1#2})}%
	\begin{assumption*}%
		\protected@edef\@currentlabel{(\textbf{#1#2})}%
	}
	{%
	\end{assumption*}
}

\newcounter{mylabelcounter}
\newcommand{\labelText}[2]{%
#1\refstepcounter{mylabelcounter}%
\immediate\write\@auxout{%
  \string\newlabel{#2}{{1}{\thepage}{{\unexpanded{#1}}}{mylabelcounter.\number\value{mylabelcounter}}{}}%
}%
}

\newbox\secondabstractbox

\newenvironment{secondabstract}[1][Second abstract]{%
  \global\setbox\secondabstractbox=\vtop\bgroup
  \normalfont\Small
  \list{}{\labelwidth\z@
    \leftmargin3pc
    \rightmargin\leftmargin
    \listparindent\normalparindent
    \itemindent\z@
    \parsep\z@ \@plus\p@
    
  }%
  \item[\hskip\labelsep\scshape #1.]%
}{%
  \endlist\egroup
}

\let\original@setabstracta\@setabstracta
\def\@setabstracta{%
  \original@setabstracta
  \ifvoid\secondabstractbox
  \else
    \vskip 10pt
    \box\secondabstractbox
    \prevdepth\z@
  \fi
}

\makeatother

\begin{document}

\begin{abstract}
	In this paper, we develop a universal method that identifies the (non-local) governing evolution equations for Continuous Time Random Walks' (CTRWs) limit processes. Given one of these processes, our method provides the form of a non-local operator, acting on space and time variables jointly, such that the (generalized) harmonic problem associated with it represents an evolution governing equation for this process. Then, the well-posedness of this problem must be established case by case. In this paper, we establish well-posedness when the process is a Feller process (on a \red{locally compact separable metric space}  $E$) time-changed with the overshooting of a subordinator. Also, we will show how our method applies to several cases when the equation and its well-posedness are already known, hence unifying several different approaches in the literature.
\end{abstract}

\begin{secondabstract}[Résumé]
Dans cet article, nous développons une méthode universelle permettant d’identifier les équations d’évolution non locales régissant les processus limites de marches aléatoires à temps continu (CTRW, pour \emph{Continuous Time Random Walks}). Pour chacun de ces processus, notre méthode détermine la forme d’un opérateur non local agissant conjointement sur les variables d’espace et de temps, de sorte que le problème harmonique généralisé qui lui est associé constitue une équation d’évolution régissant ce processus. Le caractère bien posé de ce problème doit ensuite être établi au cas par cas. Dans le présent article, nous démontrons que ce problème est bien posé lorsque le processus considéré est un processus de Feller sur un espace métrique localement compact et séparable $E$, changé de temps par le processus de dépassement associé à un subordinateur. Nous montrons également comment notre méthode s’applique à plusieurs cas dans lesquels l’équation est déjà connue et son caractère bien posé déjà établi, ce qui permet d’unifier diverses approches proposées dans la littérature.
\end{secondabstract}

\maketitle

{\hypersetup{linkcolor=black}
}

\section{Introduction}
Continuous time random walks (CTRWs) are random walks with random waiting times $(W_n)_{n \in \mathbb{N}}$ separating random jumps $(J_n)_{n \in \mathbb{N}}$. These processes have been applied in the physics literature as they are a simple model for anomalous diffusion processes, i.e., they have a mean square displacement (variance) that behaves not linearly in time. Typically, the long waiting times (heavy-tailed distributions) imply subdiffusion while the long jumps imply superdiffusion (see \cite{METZLER20001} and references therein). For possible applications and numerical methods see, e.g., \cite{barkai0, barkai1, berkowitz2006modeling, daniel24, fedotov1, fedotov2007migration,  fedotov3, fedotov2, gianni, raberto2002waiting, Scher1973a, Scher1973b, Scher1975, schumer2003fractal}. Even though the dependence between $(W_n)_n$ and $(J_n)_n$ is allowed by the theory, it is typical in statistical physics to consider CTRWs where the random variables $W_n$ and $J_n$ do not depend on the past $W_k$, $J_k$, $k=1, \dots, n-1$. However, since the distribution of the waiting times is, in general, not an exponential random variable, CTRWs are not, in general, Markov processes as they do not enjoy the lack of memory of the exponential distribution (they enjoy the semi-Markov property, see \cite[Chapter 3]{gikhman2004theory} for details).

The scaling limits of continuous time random walks are time-changed Markov processes \cite{kobayashi2011stochastic, kolokoltsov2009generalized, Meerschaert2004, straka2011lagging}; typically these processes are not Markovian as the random time-change induces intervals of constancy with arbitrary distribution; however, they can be embedded in a finite-dimensional Markov process \cite{Meerschaert2014}. One of the main tools for the analysis of the distribution of these processes, which is relevant for applications, is their governing equation, i.e., an evolution equation for their one-dimensional marginal. This equation is non-local in the time variable and reduces to a time-fractional equation in particular cases. There is a huge literature on this aspect; the goal is to find an evolution integro-differential equation for the function $q(x,t) \coloneqq \mathds{E}^xu(X_t)$ where $X_t$ is the scaling limit of a CTRW (as, for example, in \cite{ascione2021time, Baeumer2001, mirko, chen, hernandez2017, kochubei, kolokoltsov2015,  meerschaert2002governing, savtoa, STRAKA2018451, Ton19jmaa}). Despite the literature being very big, most of the results are dedicated only to the uncoupled case, i.e., when $J_n$ and $W_n$ are independent: this induces, in the limit, Markov processes time-changed with independent inverse subordinators. A fully coupled case was recently considered in \cite{ascione2024}. 
A more unifying theory is proposed in \cite{baeumer2017fokker} where the authors found a general integro-differential equation satisfied by CTRWs' limit processes. However, the equation is not an evolution equation on variables $(x,t)$ for the function $q(x,t)$, with an initial condition $q(x,0)=u(x)$; it is indeed an integro-differential equation for a functional of $q(x,t)$ where a further variable, related to the time-change, appears (also the corresponding forward equation is considered). From this approach, the function $q(x,t)$ can be obtained only in the limit sense \red{introducing a sequence of time-dependent source terms approximating the Dirac delta. From the point of view of applications, having a closed equation satisfied directly by $\mathds{E}^xu(X_t)$, which is the goal of the present paper, makes it possible to compute the observable of interest by solving a single integro-differential problem, without having to control an additional limiting procedure, either theoretically or numerically.
Merely analytically, our direct characterization provides (once well-posedness is established) existence, uniqueness, and a probabilistic representation of the solution within the relevant functional class for the problem with initial datum $u$. By contrast, when the function is obtained only as a limit of solutions, regularity, membership in the domain of the operator, and uniqueness of the limiting problem do not follow automatically.}

Our aim in this paper is, therefore, to provide a unifying general approach to this problem. We will construct a ``universal algorithm" that identifies an integro-differential equation for $q(x,t)$, having the form of the (generalized) harmonic problem
\begin{align}
    \mathfrak{A}q(x,t) \, &= \, 0, \qquad t>0,  \label{opintro}\\
    q(x,0)\, &= \, u(x), \notag 
\end{align}
where $\mathfrak{A}$ is a non-local operator acting on the variables $x$ and $t$ and $u(x)$ is a suitable initial condition (which is done in Section \ref{sec:harmonic}).
Our approach is general and can be applied to any limit process (in the class proposed in \cite{Meerschaert2014}): \red{the approach is, from the probabilistic point of view, fully rigorous (see Theorem \ref{t:1151} and Proposition \ref{p:1333}) and} it provides the form of the operator in \eqref{opintro}; however the well-posedness of the corresponding problem must be dealt with case by case. In this paper we show how to deal with the well-posedness in a situation when the governing equation for $q(x,t)$ was unknown, i.e., for a Feller process $(M_t)_{t\ge0}$ on \red{a locally compact separable metric} space $E$, time-changed with the overshooting process $(D_t)_{t\ge0}$ of an independent subordinator. \red{These results constitute the other fully rigorous part of our work: deriving the equation in this case is done in Section \ref{sec:overshoot-0}, while in Section \ref{sectionmain} we deal with well-posedness.} We observe that our analysis works when the vector variable $x$ is in a general \red{locally compact separable metric} space $E$, and not only for $x \in \mathbb{R}^d$. \red{This well-posedness result is also of independent interest, as it complements the recent results of \cite{ascione2024}, where a specific case was considered; namely, a Feller process time-changed by the undershooting of an independent subordinator. Together, the two results therefore provide a complete picture of the undershooting and overshooting cases.}

Then, we also show how our approach applies to well-known cases (see Section \ref{sec:ex}) and thus we re-obtain, using our universal methodology, fractional-type equations studied in previous literature and also the coupled case considered in \cite{ascione2024}.

\section{Harmonic functions approach for time non-local equations}
\label{sec:harmonic}
In this section, we introduce the details of our approach to understand the governing equation of a time-changed Markov process arising as a limit of a CTRW. We first recall the foundation of the CTRWs limit theory.
\subsection{Scaling limits of CTRWs and time-changed Markov processes}
\label{subsecctrw}
Let $c>0$ be a scaling parameter and define the following discrete-time Markov chain in $\mathbb{R}^d \times \R$:
\begin{align}
    (S_n^c, T_n^c) \, = \, (A_0^c, \sigma_0^c) + \sum_{k=1}^n (J_k^c, W_k^c).
\end{align}
The random variables $S_n^c$, $n\in \N$, represent the position of a particle after $n$ jumps, and $T_n^c$, $n\in\N$, are the times when the particle arrives in those positions (the epochs). \red{Here, the waiting times $W^c_n$ are assumed to be strictly positive but we allow negative values for the starting position  $\sigma^c_0$, i.e. $\sigma^c_0\in \R$.} Let
\begin{align}
    N^c_t\coloneqq \, \max \{ n \in \mathbb{N}: T_n^c \leq t  \},\quad \red{t\in \R},
\end{align}
\red{with $N^c_t=0$ of $t<T_1$.}
The CTRW is then defined as
\begin{align}
    X^c_t \coloneqq S^c_{N^c_t}
\end{align}
while the Overshooting CTRW (OCTRW) as
\begin{align}
    Y^c_t \coloneqq S^c_{N^c_t+1}.
\end{align}
The sequence $(J_k^c,W_k^c)$, $k\in\N$, is often assumed to be i.i.d. as they represent the motion of a jumping particle in the homogeneous in time and space case, however this does not mean that $J_k^c$ and $W_k^c$ have to be independent. The (O)CTRW is said to be coupled if $J_k^c$ and $W_k^c$ are dependent. The difference between the CTRW and the OCTRW stays in the dependence structure between jumps and intervals of constancy: for a CTRW the intervals of constancy are dependent on the subsequent jump, while for the OCTRW they depend on the preceding jump. Use the symbol $\Longrightarrow$ to denote the weak convergence of probability measures on $D[0, +\infty)$, the space of càdlàg functions, in the $J_1$ Skorohod topology, and suppose that
\begin{align}
    (S_{[cu]}^c, T_{[cu]}^c) \Longrightarrow (A_u, \sigma_u),\quad \text{as $c \to +\infty$},
    \label{asslimitctrw}
\end{align}
where $(A_u, \sigma_u)_u$ is a canonical Feller process on $\mathbb{R}^{d+1}$ on the filtered family of probability spaces $(\Omega, \mathcal{F}, \mathcal{F}_u, \mathds{P}^{(x,v)})$, in the sense of \cite[III.2]{revuz2013continuous}, with the second coordinate $\sigma_u$ which is strictly increasing and unbounded $\mathds{P}^{(x,v)}$-a.s. for any $(x,v) \in \mathbb{R}^{d+1}$. In this case, by \cite[Theorem 3.6]{straka2011lagging}\footnote{The article \cite{straka2011lagging} assumes that $(A_u,\sigma_u)$ is a L\'evy process, but this assumption is irrelevant for obtaining Theorem 3.6 which proof holds even in the case when $(A_u, \sigma_u)$ is a Feller process with strictly increasing and unbounded second coordinate. \red{This can be seen by observing that the L\'evy assumption is never used in Theorem 3.6 or the related lemmas: this was already pointed out in \cite{Meerschaert2014}, see the discussion around Eq. (2.3).}}, as $c \to +\infty$ it holds that
\begin{align}
  &  X_t^c  \Longrightarrow (A_{L_t-} )^+ \coloneqq \lim_{\delta\downarrow 0 } \lim_{h\uparrow 0} A_{L_{t+\delta}+h} \label{ctrwlim} \\
  & Y_t^c  \Longrightarrow  A_{L_t} \label{octrwlim}
\end{align}
where, for $t \in \mathbb{R}$,
\begin{align}
    L_t \coloneqq \inf \{ s >0 : \sigma_\red{s}>t \}.
    \label{inverse}
\end{align}
In applications in statistical physics, it is typical (see for example \cite{henry2010fractional, Meerschaert2014} and references therein) that the process $(A_u, \sigma_u)_u$ is a jump-diffusion, i.e., the associated Feller semigroup admits a generator $(\mathcal{A}, \mathcal{D}(\mathcal{A}))$ such that it has a Courrège-type form
\begin{align}
    &\mathcal{A}f(x,t)=\sum_{i=1}^d b_i(x,t) \partial_{x_i} f(x,t) + \gamma(x,t) \partial_t f(x,t) + \frac{1}{2} \sum_{1 \leq i, j \leq d} a_{ij}(x,t) \partial_{x_ix_j}^2f(x,t) \notag \\
    &+ \int \left( f(x+y, t+w) - f(x,t)-\sum_{i=1}^d y_i \mathds{1}_{(y,w) \in [-1,1]^{d+1}} \partial_{x_i} f(x,t)\right) K(x,t;dy,dw),
    \label{genctrwlim}
\end{align}
on some suitable operator core for $\mathcal{A}$. The kernel $K(x,t; \cdot, \cdot)$ is a measure on $\R^{d+1}$ and is called a jump kernel. In our case it is supported on $\mathbb{R}^d \times [t, +\infty)$ since $\sigma_u$, $u \geq 0$, is strictly increasing. The real-valued coefficients $b,\gamma,a_{ij}$ are typically assumed to satisfy Lipschitz and linear growth conditions to induce a solvable  SDE (see, e.g., \cite[Section 6.2]{applebaum}).

\subsection{Governing equations and generalized harmonicity}\label{ss:general}
Our aim is to obtain governing equations for the processes $(A_{L_t})_{t\ge0}$ and $\big((A_{L_t-})^+\big)_{t\ge0}$, defined as in \eqref{ctrwlim} and \eqref{octrwlim}, in full generality. Let now $(A_u, \sigma_u)_u$ be a general Feller process (possibly sub-Markovian\footnote{\red{This means that the process $(A_u,\sigma_u)_u$ may be killed at some random time and sent to the cemetery state. Analytically, this is reflected in the fact that its semigroup may satisfy $P_u \1(x) < 1$, i.e. the semigroup can lose mass. Note, in contrast, that the process generated by $\mathcal{A}$ as in \eqref{genctrwlim} is not sub-Markovian since the generator does not include a killing term.}}) on $E\times \R$, with the corresponding family of probability measures $\pr^{(x,v)}$, such that the second coordinate $\sigma_u$ is strictly increasing and unbounded, and where $E$ is locally compact separable metric space.

Since $\sigma_u$ strictly increases, for its crossing time of level $t$, that is $L_t$ defined in \eqref{inverse}, it holds that $L_t=\wt L_t\coloneqq \inf\{u\ge0: \sigma_u\ge t\}$. Therefore, the random time $L_t$ is also the exit time of the Markov process $(A_u, \sigma_u)_u$, from the open set $E \times (-\infty, t)$. Hence, one could expect that the process stopped at $L_t$ has a harmonic property, i.e., the equation $\mathcal{A}q(x,v)=0$, subject to $q(x,v)=u(x,v)$ for $(x,v) \in E \times [t, +\infty)$ is satisfied by $q(x,v) =\mathds{E}^{(x,v)}u(A_{L_t}, \sigma_{L_t})$. However, this is not an evolution as desired, since it is in the variables $(x,v)$. We instead want a governing equation acting on the space-time variables $(x,t)\in E\times (0,{+\infty})$ and with a suitable initial condition on $E\times \{0\}$. Therefore, we modify the previous classical approach as follows.\footnote{For the reader's convenience, it is helpful to study first the case when $(A_u,\sigma_u)_{u\ge0}$ is Markov additive, presented in the next subsection, which has much less technical construction.}

Let  $p_{u}(x,v;dy,dw)$ denote the transition functions/kernels of the canonical Feller process $(A_u, \sigma_u)_u$, that is $p_{u}(x,v;dy,dw)=\pr^{(x,v)}(A_u\in dy,\sigma_u\in dw)$. For each $t\ge0$, define the following transition functions on the state space $E\times[0,+\infty)$ equipped with the Borel sigma-algebra $\mathcal{B}$: 
\begin{align}
 \begin{split}\label{transitionover}
     \,^t{p}_{u}^+(x,v;dy,dw) &\coloneqq \mathds{1}_{[0<w\leq v]}  p_{u}(x,t-v;dy,t-dw)\\&\hspace{8em}+
 \delta_0(dw)\,^t\upvarpi_{u}^+(x,t-v,dy),  
 \end{split}\\
 \begin{split}\label{transitionunder}
     \,^t{p}_{u}^-(x,v;dy,dw) &\coloneqq \mathds{1}_{[0<w\leq v]}  p_{u}(x,t-v;dy,t-dw)\\
     &\hspace{8em}+\delta_0(dw)\,^t\upvarpi_{u}^-(x,t-v,dy).
 \end{split}
\end{align}
Here, $\,^t\upvarpi_{u}^+(x,t-v,dy)\coloneqq \pr^{(x,t-v)}\left( A_{L_t}\in dy, \sigma_u\in[t,{+\infty})\right)$ and $\,^t\upvarpi_{u}^-(x,t-v,dy)\coloneqq \pr^{(x,t-v)}\left( A_{L_t-}\in dy, 
\sigma_u\in[t,{+\infty})\right)$. Indeed, it can be checked that for each $t\ge 0$, the kernels $\,^tp^{\pm}$ satisfy the Chapman-Kolmogorov property, and thus forming (sub-)Markov transition functions.

By this construction, by putting $v=t$, we get
\begin{align}
    \begin{split}
        ^t p^+_u(x,t;dy,dw) & =  \mathds{1}_{[w>0]} \mathds{P}^{(x,0)} \left( A_u \in dy, t-\sigma_u \in dw  \right) \\ &\hspace{4em}+\mathds{P}^{(x,0)} \left( A_{L_t} \in dy, L_t \leq u \right)\delta_0(dw),
    \end{split}\\
    \begin{split}
        ^tp^-_u(x,t;dy, dw)&= \mathds{1}_{[w>0]} \mathds{P}^{(x,0)} \left( A_u \in dy, t-\sigma_u \in dw \right)\\ &\hspace{4em}+ \mathds{P}^{(x,0)} \left( A_{L_t-} \in dy, L_t \leq u \right) \delta_0(dw).
    \end{split}
\end{align}

Now, for each $t\ge 0$, we can define the processes $^tZ^+=(^tA_u^+, \,^t\sigma_u^+)_u$, and $^tZ^-=(^tA_u^-, \,^t\sigma^-_u)_u$, as canonical processes (as in \cite[III.2]{revuz2013continuous}) with the state space $E \times [0, +\infty)$ and the transition kernels $^tp^+$ and $^tp^-$, respectively. We will denote the corresponding families of probability measures by $^t\mathds{P}_{\pm}^{(x,v)}(\cdot)$, $(x,v) \in E \times [0, +\infty)$.

A  description of the processes $^tZ^+$ and $^tZ^-$ is the following. When the processes $^tZ^+$ and $^tZ^-$ start from $(x,t)$ and when $(A_u, \sigma_u)$ starts from $(x,0)$, then one has
\begin{align}
    ^tZ^+_u\, = \, \begin{cases}
        (A_u, t-\sigma_u), \quad & u < L_t, \\ (A_{L_t},0), & u \geq L_t,
    \end{cases}, \quad ^tZ^-_u = \begin{cases}
        (A_u, t-\sigma_u), \quad & u < L_t, \\ (A_{L_t-},0), & u \geq L_t,
    \end{cases} 
\end{align}
\red{in the sense of finite-dimensional distributions, but also on the whole paths since both processes are constructed canonically with càdlàg realizations using the same transition function.}
Since we want to obtain a governing equation for the processes defined in \eqref{ctrwlim} and \eqref{octrwlim}, by defining
\begin{align}\label{exittime}
    \tau^{\pm}_0 \coloneqq \inf \{ u \geq 0 : \,^tZ^{\pm}_u \in E \times \{ 0 \} \},
\end{align}
we have,
\begin{align*}
   & \mathds{P}^{(x,0)} \left( A_{L_t} \in dy \right) =  \,^t\mathds{P}_+^{(x,t)} \left( \,^tA_{\,\tau_0^+}^+ \in dy \right),\\
   & \mathds{P}^{(x,0)} \left( A_{L_t-} \in dy \right) =  \,^t\mathds{P}_-^{(x,t)} \left( \,^tA_{\,\tau_0^-}^- \in dy \right).
\end{align*}

By following the arguments above, we obtain the following statement.
\begin{theorem}\label{t:1151}
    Let $(A_u, \sigma_u)_u$ be a Feller process as above, with the transition probabilities $p_{u}(x,v;dy,dw)$. Then, for any $t\ge0$, there exist two Markov processes $\,^tZ^+=(^tA_u^+, \,^t\sigma_u^+)_u$ and $\,^tZ^-=(^tA_u^-, \,^t\sigma_u^-)_u$, on the canonical probability space with measures $\,^t\mathds{P}_\pm^{(x,t)}$, with the transition probabilities 
    \begin{align*}
 \,^t{p}_{u}^+(x,v;dy,dw) &\coloneqq \mathds{1}_{[0<w\leq v]}  p_{u}(x,t-v;dy,t-dw)+
 \delta_0(dw)\,^t\upvarpi_{u}^+(x,t-v,dy), \\
 \,^t{p}_{u}^-(x,v;dy,dw) &\coloneqq \mathds{1}_{[0<w\leq v]}  p_{u}(x,t-v;dy,t-dw)+\delta_0(dw)\,^t\upvarpi_{u}^-(x,t-v,dy),
\end{align*}
respectively, where $\,^t\upvarpi_{u}^+(x,t-v,dy)\coloneqq \pr^{(x,t-v)}\left( A_{L_t}\in dy, \sigma_u\in[t,{+\infty})\right)$ and $\,^t\upvarpi_{u}^-(x,t-v,dy)\coloneqq \pr^{(x,t-v)}\left( A_{L_t-}\in dy, 
\sigma_u\in[t,{+\infty})\right)$, such that 
\begin{align}
   & \mathds{P}^{(x,0)} \left( A_{L_t} \in dy \right) =  \,^t\mathds{P}_+^{(x,t)} \left( \,^tA_{\,\tau_0^+}^+ \in dy \right), \label{317}\\
   & \mathds{P}^{(x,0)} \left( A_{L_t-} \in dy \right) =  \,^t\mathds{P}_-^{(x,t)} \left( \,^tA_{\,\tau_0^-}^- \in dy \right). \label{318}
\end{align}
where $\tau^{\pm}_0 \coloneqq \inf \{ u \geq 0 : \,^tZ^{\pm}_u \in E \times \{ 0 \} \}$.
\end{theorem}

It is a classical result in CTRWs limits theory that, in the uncoupled case, the processes \eqref{ctrwlim} and \eqref{octrwlim} coincide a.s. as $A_u$ and $\sigma_u$ do not have simultaneous jumps, see e.g. \cite[Section 3]{straka2011lagging}. Then, using our approach, this becomes clear from equations \eqref{317} and \eqref{318} by a simple conditioning argument. 

Denote now by $^t\mathfrak{A}^{\pm}$ the generators of the processes $^tZ^{\pm}$. The family of harmonic problems for these processes, exiting from the open set $E \times (0, +\infty)$ would be
\begin{align}
    \begin{cases} ^t\mathfrak{A}^{\pm} q_t(x,v) \, = \, 0 \qquad &(x,v) \in E \times (0, +\infty) \\ q_t(x,v) = u(x) & (x,v) \in E \times \{ 0 \},
\end{cases}
\label{1407}
\end{align}
which, heuristically, should be solved by
\begin{align}
    q_t(x,v) \, = \, \red{\,^t}\mathds{E}_{\pm}^{(x,v)} u(\,^tA^{\pm}_{\red{\tau_{0}^{\pm}}}).
    \label{1408}
\end{align}
In view of \eqref{317} and \eqref{318}, we should obtain
\begin{align}
 &     q_t(x,t) \, = \red{\,^t}\mathds{E}_{+}^{(x,t)} u(\, ^tA^{+}_{\tau_{0}^{+}}) \, = \, \mathds{E}^{(x,0)} u(A_{L_t}) \label{320}\\ &   q_t(x,t) \, =\red{\,^t}\mathds{E}_{-}^{(x,t)} u(\, ^tA^{-}_{\tau_{0}^{-}}) \, = \, \mathds{E}^{(x,0)} u(A_{L_t-}).\label{321}
\end{align}
If \eqref{1407} is indeed satisfied by \eqref{1408} then this means that the governing equation of the CTRWs' limit process has this form.

{
With this approach, it is possible to recover several well-known cases in the literature: the general ``algorithm'' is outlined in Subsection \ref{ss:harmonic approach algorithm}, a new governing equation is studied in Sections \ref{sec:overshoot-0} and \ref{sectionmain}, and specific known instances are treated in Section \ref{sec:ex}.

It is also important to note that our approach connects the process $\,^tZ^-$  to $A_{L_t-}$ via \eqref{318}. Under the setting of the following proposition, this is equivalent to studying $(A_{L_t-})^+$, which is the process arising naturally as a scaling limit of the CTRW.}
\begin{proposition}\label{p:1333}
    {For every Feller process $(A_u, \sigma_u)_u$ as above it holds that
    \begin{align}\label{eq:1722}
        \pr^{(x,s)}((A_{L_t-})^+\ne A_{L_t-}) &= \pr^{(x,s)}(\exists h\ge 0: \Delta A_h>0, {\Delta\sigma_h>0}, \sigma_h=t),
    \end{align}
    for all $(x,s)\in E\times (-\infty,t]$.
    Suppose, furthermore, that $E=\R^d$ and that $(A_u,\sigma_u)_u$ is generated by the operator as in \eqref{genctrwlim} with core $C_c^\infty(\R^{d+1})$. Then, for all $t>0$, the probability in \eqref{eq:1722} equals zero for all $x\in \R^d$ and $s\le t$, if and only if, for all $t>0$ and for almost every $h>0$, it holds that  $K(A_h,\sigma_h;\R^d\times \{t-\sigma_h\})=0$, $\pr^{(x,s)}$-a.s.\ for all $x\in \R^d$ and $s\le t$.}
\end{proposition}
\begin{proof}
    First, note that, pathwise, if $A$ is continuous at $u=L_t$, then $(A_{L_t-})^+= A_{L_t-}$ since $L_t$ is continuous. If $A$ is discontinuous at $u=L_t$ and $L_{t+\varepsilon}=L_t$ for all sufficiently small $\varepsilon>0$, then $(A_{L_t-})^+= A_{L_t-}$. On the other hand, if $A$ is discontinuous at $u=L_t$ and $L_{t+\varepsilon}>L_t$ for all $\varepsilon>0$, then $(A_{L_t-})^+= A_{L_t}\ne A_{L_t-}$. {However, on the set $\{\Delta\sigma_{L_t}=0\}$ (i.e., $L_t$ is not a jump time of $\sigma$) we have $A_{L_t-} = A_{L_t}$ $\pr^{(x,s)}$-a.s. by the following calculation. Indeed, on $\{\Delta\sigma_{L_t}=0\}$, for the stopping times $L_{t-\epsilon}$ and $L_t$ it holds that $L_{t-\varepsilon} < L_t$ for all $\epsilon$ small enough and $L_{t-\epsilon} \nearrow  L_t$ so $A_{L_t-}= \lim_{\epsilon \to 0}A_{L_{t-\epsilon}} = A_{L_t}$ by the quasi-left-continuity of the Feller process $(A_u,\sigma_u)_u$ (which is consequently a Hunt process). Therefore, the only case in which $(A_{L_t-})^+\ne A_{L_t-}$ for fixed $t>0$ with possibly positive probability is if $\sigma$ hits $t>0$ with a jump and $A$ jumps simultaneously. Thus \eqref{eq:1722} is verified.}
    
    To deal with \eqref{eq:1722} when $E=\R^d$ and when $(A_u,\sigma_u)_u$ is generated by the operator as in \eqref{genctrwlim} with core $C_c^\infty(\R^{d+1})$, recall that the process $(A_u,\sigma_u)_u$ is Feller, hence a Hunt process, so it admits a L\'evy system $(\mathcal{K},\mathcal{H})$. Here $\mathcal{K}$ is a kernel on $\R^{d+1}$ and $\mathcal{H}$ is an additive functional. In other words, if we denote $\mathcal{X}_u=(A_u,\sigma_u)$, $u\ge 0$, then for all $f:\R^{d+1}\times\R^{d+1} \to [0,+\infty)$ that vanish at the diagonal and for all stopping times $T$, it holds that
    \begin{align}\label{eq:1719}
        \EE^{(x,s)}\sum_{h\le T} f(\mathcal{X}_{h-},\mathcal{X}_h)=\EE^{(x,s)}\int_0^T\int_{\R^{d+1}} f(\mathcal{X}_h,y)\mathcal{K}(\mathcal{X}_h;dy)d\mathcal{H}_h.
    \end{align}
    By \cite[Theorem 3.5]{Schilling_PTRF}, it holds that $\mathcal{H}_h=h$ and $\mathcal{K}=K$, i.e. $\mathcal{K}$ is the jumping kernel $K$ of $(A,\sigma)$.

    For $x\in \R^{d+1}$ write $x=(a,s)$ where $a\in \R^d$ and $s\in \R$. Then pick
    \begin{align*}
        f((a_1,s_1),(a_2,s_2))\coloneqq \1_{a_1\ne a_2} \1_{s_2=t},
    \end{align*}
    and $T={+\infty}$ to get from \eqref{eq:1722} and \eqref{eq:1719} that
    \begin{align*}
        \pr^{(x,s)}((A_{L_t-})^+\ne A_{L_t-}) &= \EE^{(x,s)}\int_0^{+\infty}\int_{\R^{d+1}} f(\mathcal{X}_h,y){K}(\mathcal{X}_h;dy)\,dh.\\
        &=\EE^{(x,s)}\int_0^{+\infty} K(A_h,\sigma_h;\R^{d+1}\times \{t-\sigma_h\})dh.
    \end{align*}
    The claim now follows directly.
\end{proof}

\begin{remark}[{{Counterexample for $(A_{L_t-})^+=A_{L_t-}$ a.s.}}]
    In general, it may happen that $(A_{L_t-})^+=A_{L_t-}$, $\pr^{(x,s)}$-a.s. for all $x\in \R^d$ and $s\le t$, does not hold.

    Indeed, fix $t>0$. First we show that
    \begin{align}\label{eq:1610}
        G_{|C_c^\infty(\R)}f(x)=f'(x)+F(x)\big(f(t)-f(x)\big),\quad x\in \R,
    \end{align}
    with $F\in C_c^\infty(\R)$ that satisfies $F\ge 0$, $F\not \equiv 0$ and $F(x)=0$, $x\ge t$, is a Feller generator. To this end, note that $f\mapsto f'$ is the pure-drift operator, hence a Feller generator, while $f\mapsto Lf\coloneqq F(\cdot)\big(f(t)-f(\cdot)\big)$ is a bounded operator from $C_0(\R)$ to $C_0(\R)$. By perturbation theory, see e.g. \cite[Theorem 6.1.3]{applebaum-semigroups}, $G$ is a Feller semigroup on $C_0(\R)$. This generator is associated with the strictly increasing process $(\sigma_u)_u$ with the following behavior of trajectories. Assume that we start from  $x< \max\{\supp{F}\}$. Then, we move as a pure drift while having the opportunity to jump exactly to $t$ with the position-dependent intensity given by the function  $F$. In other words, since $F\not\equiv 0$, we have a positive probability of reaching $t$ at some point. If we start from $x\ge \max\{\supp{F}\}$, the process moves as a pure drift.

    Now we define the process $(A_u,\sigma_u)_u$ in $\R^2$ by the following relation: for $f\in C_0(\R^2)$ let
    \begin{align}\label{eq:1613}
        P_uf(x,y)\coloneqq \ex^{(x,y)}f(A_u,\sigma_u)\coloneqq \ex^{y}f(x-y+\sigma_u,\sigma_u),\quad x,y\in \R.
    \end{align}
    Here, $(\sigma_u)_u$ is the process generated by $G$ as in \eqref{eq:1610}, where its expectation operator is denoted by $\ex^{y}$. Now we need to prove that \eqref{eq:1613} is in fact a Feller semigroup on $\R^2$. Since $(\sigma_u)_u$ is Feller, it is easy to show that  $P_tf$ vanishes at infinity for $f\in C_0(\R^2)$. To show continuity, observe that 
    \begin{align}\label{eq:1624}
        P_tf(x,y)=\int_\R f(x-y+w,w)\pr^y(\sigma_u\in dw).
    \end{align}

    Continuity now follows from the generalized dominated convergence theorem: if $f_n\to f$ uniformly on compact subsets of a Polish space $\mathbb S$, the $f_n$ are uniformly bounded, and $\mu_n\to \mu$ weakly where $\mu_n$ are finite measures, then $\int_E f_n d\mu_n\to \int_E fd\mu$, which is easily proved using Prokhorov's theorem. We may use this claim since $(\sigma_u)_u$ is Feller, so the map $y\mapsto \pr^y(\sigma_u\in dw)$ is weakly continuous.

    Now note that $A$ and $\sigma$ have exactly the same jump time, and the second coordinate $(\sigma_u)_u$ always jumps to the value $t$. Since $F\not \equiv 0$, for $y\ll t$, the probability of $\sigma$ jumping to $t$ is strictly positive under $\pr^y$. Hence for $y<<t$ we have
    \begin{align}
       & \pr^{(x,y)}((A_{L_t-})^+\ne A_{L_t-})=\pr^{(x,y)}(\exists s\ge0:  \Delta A_s>0, {\Delta \sigma_s>0}, \sigma_s=t)\\
        &=\pr^{(x,y)}(\exists s\ge0:  \Delta \sigma_s>0, \sigma_s=t)>0,
    \end{align}
    for all $x\in \R$.

    This example can be easily generalized to the case where $(A_{L_t-})^+\ne A_{L_t-}$ for infinitely many $t>0$.
    \end{remark}

\red{The general harmonic approach for the construction of governing equations presented in this subsection which leads to the operators $\,^t\mathfrak{A}^\pm$ may be cumbersome to digest at first. Therefore, in the next subsection we discuss the case when $(A_u,\sigma_u)_u$, is a Markov additive process which simplifies the previous construction.}

\subsection{Harmonicity problem for Markov additive $(A_u,\sigma_u)$}\label{ss:additive}
    Probabilistically, the case when $(A_u,\sigma_u)_u$ is Markov additive is the case when the future evolution of $(A_u,\sigma_u)$ depends only on the current position of $A_u$. Furthermore, the first coordinate of $(A_u,\sigma_u)_{u\ge 0}$ is a Markov process by itself, and the second one is homogeneous in space. For details, see, e.g., \cite{cinlarma}. Analytically, if we have the generator representation \eqref{genctrwlim}, then the coefficients $b_i$, $\gamma$, $a_{ij}$, and the jumping kernel $K$ do not depend on $t$. This allows us to simplify the previous considerations in which we can abandon the parameter $t$ in the transitions $\,^tp^{\pm}$ and the operators $\,^t\mathfrak{A}^{\pm}$, while at the same time we still cover a wide range of interesting examples.
    
    Indeed, the homogeneity of the second coordinate says $p_u(x,t-v;dy,t-dw)=p_u(x,-v;dy,-dw)=p_u(x,0;dy,v-dw)$, so to obtain the evolution of $(A_{L_t})_t$, we can define only one family of transition functions. Instead of \eqref{transitionover}, we can look at
    \begin{align}
        {p}_{u}^+(x,t;dy,dw) &\coloneqq \mathds{1}_{[0<w\leq t]}  p_{u}(x,0;dy,t-dw)\nonumber\\
 &\hspace{4em}+ \delta_0(dw)p_{u} (x,0; dy, [t, +\infty)).\label{transitionover-NO2}
    \end{align}
    This is a (sub-)Markov transition kernel, and it holds that 
    \begin{align}
        {p}_{u}^+(x,t;dy,dw)=\pr^{(x,0)}(A_u\in dy,(t-\sigma_u)\vee 0\in dw).
    \end{align}
    Note that, since $A_u$ is Markov by itself, we allowed the first coordinate to keep moving after stopping the second one, unlike in \eqref{transitionover} where we have stopped both coordinates.

    This construction gives us a more explicit version of Theorem \ref{t:1151} in the Markov additive case.
    \begin{proposition}
         Let $(A_u, \sigma_u)_u$ be a Markov additive process as above, with the transition probabilities $p_{u}(x,v;dy,dw)$. Then, there exist a Markov process $Z^+=(A^+_u,\sigma^+_u)_u$, on the canonical probability space with the measures $\mathds{P}_+^{(x,t)}$, with the transition probabilities 
    \begin{align}\label{transitionunder-NO4}
        {p}_{u}^+(x,t;dy,dw)=\pr^{(x,0)}(A_u\in dy,(t-\sigma_u)\vee 0\in dw).
    \end{align}
such that 
\begin{align}
   & \mathds{P}^{(x,0)} \left( A_{L_t} \in dy \right) =  \mathds{P}_+^{(x,t)} \left(A_{\,\tau_0^+}^+ \in dy \right),
\end{align}
where $\tau^{+}_0 \coloneqq \inf \{ u \geq 0 : Z^{+}_u \in E \times \{ 0 \} \}$.
    \end{proposition}
    
   We should expect that  the harmonic problem, where $\mathfrak{A}^{+}$ is the generator of $Z^+$,
   \begin{align}
    \begin{cases} \mathfrak{A}^{+} q(x,t) \, = \, 0 \qquad &(x,t) \in E \times (0, +\infty) \\ q(x,t) = u(x) & (x,t) \in E \times \{ 0 \},
\end{cases}
\label{1407-a}
\end{align}
is solved by $q(x,t)=\ex_+^{(x,t)}u(A^+_{\tau^+_0})=\ex^{(x,0)} u(A_{L_t}).$ In other words, $\mathfrak{A}^{+}$ is a good candidate for the operator describing the evolution of the process $A_{L_t}$. 

We have a similar claim regarding the evolution of $A_{L_t-}$.
\begin{proposition}
         Let $(A_u, \sigma_u)_u$ be a Markov additive process as above, with the transition probabilities $p_{u}(x,v;dy,dw)$. Then, there exist a Markov process $Z^-=(A^-_u,\sigma^-_u)_u$, on the canonical probability space with the measures $\mathds{P}_-^{(x,t)}$, with the transition probabilities 
    \begin{align}
    {p}_{u}^-(x,t;dy,dw) &\coloneqq \mathds{1}_{[0<w\leq t]}  p_{u}(x,0;dy,t-dw)\nonumber\\
 &\hspace{4em}+ \delta_0(dw)\pr^{(x,0)} (A_{L_t-}\in dy, \sigma_u\in [t, +\infty)),\label{transitionunder-NO2}
\end{align}
such that 
\begin{align}
   & \mathds{P}^{(x,0)} \left( A_{L_t-} \in dy \right) =  \mathds{P}_-^{(x,t)} \left(A_{\,\tau_0^-}^- \in dy \right),
\end{align}
where $\tau^{-}_0 \coloneqq \inf \{ u \geq 0 : Z^{-}_u \in E \times \{ 0 \} \}$.
    \end{proposition}
Note that here we have
\begin{align}
\begin{split}\label{1550}
    {p}_{u}^-(x,t;dy,dw)&=\mathds{1}_{[w>0]} \mathds{P}^{(x,0)} \left( A_u \in dy, t-\sigma_u \in dw \right)\\ &\hspace{4em}+ \mathds{P}^{(x,0)} \left( A_{L_t-} \in dy, L_t \leq u \right) \delta_0(dw),
\end{split}
\end{align}
so by arguments as before, the harmonic problem
   \begin{align}
    \begin{cases} \mathfrak{A}^{-} q(x,t) \, = \, 0 \qquad &(x,t) \in E \times (0, +\infty) \\ q(x,t) = u(x) & (x,t) \in E \times \{ 0 \},
\end{cases}
\label{1407-aa}
\end{align}
 where $\mathfrak A^-$ is the generator of $Z^-$, should \red{be} solved by $q(x,t)=\ex_-^{(x,t)}u(A^-_{\tau^-_0})=\ex^{(x,0)} u(A_{L_t-})$.

This unified approach can be applied to obtain evolution equations studied in the literature but dealt with different methods. See Section \ref{sec:ex} where we discuss these evolutions in greater detail.

\begin{remark}[Courrège-type form of $\mathfrak{A}^+$]
\label{rem25}
    If the Courrège-type form of $\mathcal{A}$ is known, and e.g. given by \eqref{genctrwlim}, then it is possible to heuristically deduce a Courrège-type form of $\mathfrak{A}^+$. Indeed, recall that $\ex^{(x,t)}_+ f(Z_u^+)=\ex^{(x,0)}[ f(A_u,(t-\sigma_u)\vee 0)]$ so
    \begin{align}
        \begin{split}\label{gen:1026}
            \mathfrak{A}^+f(x,t)&=\lim_{u\searrow0}\frac{\ex^{(x,t)}_+f(Z_u^+)-f(x,t)}{u}\\&=\lim_{u\searrow0}\frac{\ex^{(x,0)}[ f(A_u,(t-\sigma_u)\vee 0)]-f(x,t)}{u}=\mathcal{A}F_t(x,0),
        \end{split}
    \end{align}
    where $F_t(x,s)=f(x,(t-s)\vee 0)$. Here we emphasize that the function $F_t$ is certainly not in the $C_0(\R^{d+1})$-domain of $\mathcal{A}$ since it does not vanish at the infinity. However, we can look at \eqref{gen:1026} heuristically, or e.g. by considering the pointwise extensions of the operators. Hence, if $\mathcal{A}$ has the form \eqref{genctrwlim}, by \eqref{gen:1026} the generator $\mathfrak{A}^+$ has the form
    \begin{align*}
        &\mathfrak{A}^+f(x,t)=\sum_{i=1}^d b_i(x) \partial_{x_i} f(x,t) - \gamma(x) \partial_t f(x,t) + \frac{1}{2} \sum_{1 \leq i, j \leq d} a_{ij}(x) \partial_{x_ix_j}^2f(x,t)  \\
    &+ \int \left( f(x+y, (t-w)\vee 0) - f(x,t)-\sum_{i=1}^d y_i \mathds{1}_{(y,w) \in [-1,1]^{d+1}} \partial_{x_i} f(x,t)\right) K(x;dy,dw).
    \end{align*}
    Note that here the jumping kernel $K(x;dy,dw)$ is independent of $t$ by the Markov additive property of $(A_u,\sigma_u)_u$.
\end{remark}

\red{\subsection{Harmonic approach to governing equations}\label{ss:harmonic approach algorithm}
The approach based on generalized harmonic functions introduced in the previous Subsections \ref{ss:general} and \ref{ss:additive} is suitable for determining the governing equations of CTRWs limit processes. The ``algorithm'' should be the following. Given a CTRW with jumps and waiting times $(J_n, W_n)$, the first step is to identify the limit $(A_u, \sigma_u)$ and the form of its generator $\mathcal{A}$ (see \eqref{genctrwlim}). The second step is to determine the transition probabilities $p^+$ and $p^-$ defined in \eqref{transitionover} and \eqref{transitionunder} (using the simplified form \eqref{transitionover-NO2} and \eqref{transitionunder-NO2} in the Markov additive case). From $p^+$ and $p^-$ one can identify the form of the generators $\mathfrak{A}^+$ and $\mathfrak{A}^-$ that gives the equations governing the evolution of $\mathds{E}^{(x,0)}u(A_{L_t})$ and $\mathds{E}^{(x,0)}u(A_{L_t-})$, respectively. Regularity and well-posedness of the problems is an issue that must be solved case by case, since a general theory of harmonic functions for generators of Feller processes is still not at hand.

Here we mention the developed theory of harmonic problems for unimodal L\'evy processes \cite{bogdan_extension}, in the pointwise, but also in the distributional and the Dirichlet forms settings. However, our process $(A_u,\sigma_u)_u$ is not symmetric, let alone unimodal L\'evy. Moreover, it is known that in some quite reasonable cases, see \cite[p. 150]{BS}, the harmonic function (i.e. the candidate for a solution) may lack sufficient regularity to apply the generator on it pointwise. In non-pointwise settings, we mention also \cite{Chen09} and \cite{MZZ10} where the harmonic problem was studied from the Dirichlet forms point of view in great generality in symmetric and non-symmetric settings, respectively.

Therefore, to demonstrate the usability of our approach, Section \ref{sec:overshoot-0} first derives the governing equation driven by a Feller process time-changed with the overshooting of an independent subordinator -- an equation that, up to now, has been unknown. Then, in Section \ref{sectionmain}, we give a rigorous treatment of the regularity and well-posedness of this equation.

Finally, Section \ref{sec:ex} shows how our harmonic approach independently recovers many governing equations already known in the literature for (O)CTRW limit processes.
}

\section{Overshooted processes and coupled non-local operators}\label{sec:overshoot-0}

In this paper, we will provide a rigorous treatment of the evolution equation for a class of processes arising as scaling limits of OCTRWs and leading to Feller processes time-changed with the overshooting of a subordinator. In particular, we will rigorously solve OCTRW harmonic problem \eqref{1407} in the case when $\mathcal{A}$ generates the process $(M_{\sigma_u-\sigma_0}, \sigma_u)_u$ (see Section \ref{exunder} for the limit of the corresponding CTRW).
Continuous time random walks leading to these kind of processes have been considered, for example, in \cite{becker2004limit, meerschaert2002governing, shlesinger1982random}. For the sake of clarity, we exemplify the theory using the case where the computation can be done quite explicitly (see Section \ref{exunder} for a more general, less explicit, approach).
\begin{example}
\label{examplecoupled}
  Suppose that the jumps $J_n$, $n \in \mathbb{N}$, and the waiting times $W_n$, $n \in \mathbb{N}$, form an i.i.d. sequence of pairs that have the joint density function 
\begin{align}
    f_{J,W}(x,s) \, = \, \frac{1}{\sqrt{2\pi s}}e^{-x^2/2s} e_\alpha (s), \qquad s > 0, x \in \mathbb{R},
    \label{jointgausmittag}
\end{align}
where $e_\alpha(\cdot)$ is the density of the Mittag-Leffler distribution, i.e., for $s>0$,
\begin{align}
    e_\alpha(s) \coloneqq -\frac{d}{ds} E_\alpha(-s^\alpha) \coloneqq -\frac{d}{ds} \sum_{k=0}^{+\infty} \frac{(-s^\alpha)^k}{\Gamma(1+\alpha k)},
\end{align}
where $\alpha\in (0,1)$. \red{Observe that $J\mid W \sim N(0,W)$}.

For a scaling parameters $c>0$, define the scaled jumps and waiting times by \[(J^{c}_k,W^{c}_k)=(c^{-1/\alpha}J_k,(\sqrt{c})^{-1/\alpha}W_k)=\red{(c^{-1/\alpha}J_k,c^{-1/(2\alpha)}W_k)}, \quad k\in\N.\] 

We now compute the Fourier-Laplace transform of $(S_{[cu]}^c,T_{[cu]}^c)$. With classical manipulations, by a simple conditioning argument and applying the freezing lemma we have that, for $\lambda >0$, $\xi \in \mathbb{R}$,
\begin{align}
   & \mathds{E}\left(e^{-\lambda c^{-1/\alpha}\sum_{k=1}^{[cu]} W_k + i\xi (\sqrt{c})^{-1/\alpha}\sum_{k=1}^{[cu]} J_k}\right) \notag \\ = \, &\mathds{E} \left(\mathds{E} \left[ e^{-\lambda c^{-1/\alpha}\sum_{k=1}^{[cu]} W_k + i\xi (\sqrt{c})^{-1/\alpha}\sum_{k=1}^{[cu]} J_k} \mid W_1, \cdots, W_{[cu]} \right]\right) \notag \\
   = \, & \mathds{E}\prod_{k=1}^{[cu]} \left(e^{-c^{-1/\alpha}\lambda W_k} \mathds{E} \left[ e^{i (\sqrt{c})^{-1/\alpha}\xi J_k} \mid W_k \right]\right) \notag \\
   = \, & \left( \mathds{E}\left[e^{-\lambda c^{-\frac{1}{\alpha}} W_1} e^{-W_1 \frac{\xi^2}{2c^{1/\alpha}}} \right]\right)^{[cu]} \notag \\
   = \, & \left( \frac{1}{1+ \frac{1}{c}(\lambda + \xi^2/2)^\alpha} \right)^{[cu]},
   \label{forosmall}
\end{align}
where in the last step we also used that (see, for example, \cite{scalas2006five})
\begin{align}
    \int_0^{+\infty} e^{-\lambda s} e_\alpha (s) ds \, = \, \frac{1}{1+\lambda^\alpha}.
\end{align}
Then, from \eqref{forosmall}, we obtain
\begin{align}
   & \mathds{E}\left(e^{-\lambda c^{-1/\alpha}\sum_{k=1}^{[cu]} W_k + i\xi (\sqrt{c})^{-1/\alpha}\sum_{k=1}^{[cu]} J_k}\right) \to e^{- u\left( \lambda + \frac{\xi^2}{2} \right)^\alpha} ,\quad\text{as $c\to+\infty$.}
    \label{foulap}
\end{align}
It is easy to see that the process $(B_{\sigma_t}, \sigma_t)_t$, where $(B_t)_t$ is a standard Brownian motion and $(\sigma_t)_t$ a $\alpha$-stable subordinator independent of $(B_t)_t$,  has the Fourier-Laplace symbol as in \eqref{foulap}, thus identifying the process in the limit  \eqref{asslimitctrw}. It follows by \eqref{ctrwlim} and \eqref{octrwlim} that the associated CTRW converges to $B_{\sigma_{L_t-}}$ while the OCTRW converges to $B_{\sigma_{L_t}}$, where $L_t$ is the inverse of the stable subordinator $(\sigma_t)_t$.
\end{example}

Now we return to the general theory. We assume that $(M_u)_{u\ge0}$ is a Feller process (possibly sub-Markov) on a locally compact separable metric space $E$ and $(\sigma_u-\sigma_0)_{u\ge0}$  is a subordinator independent of $(M_u)_{u\ge0}$. The process $(M_{\sigma_u-\sigma_0}, \sigma_u)_{u\ge 0}$, is in fact Markov additive (see \cite[Lemma 4.1]{ascione2024}) and thus the general family of harmonic problems \eqref{1407} simplifies to the harmonic problem \eqref{1407-a}, while the operators $^t\mathfrak{A}^+$ and the transitions $\,^tp^+$ simplify to  $\mathfrak{A}^+$ and $p^+$, as in Subsection \ref{ss:additive}. Recall that, by the construction, the solution to the harmonic problem \eqref{1407}, or more precisely to \eqref{1407-a}, should be
\begin{align}
    q(x,t) \, = \, \mathds{E}^{(x,0)}u(M_{\sigma_{L_t}})=\mathds{E}^{(x,0)}u(M_{D_t}).
    \label{expover}
\end{align}
Here $D_t\coloneqq \sigma_{L_t}$, and it is called the \red{overshooting} of the subordinator $\sigma_t-\sigma_0$. 

We continue by providing a precise formulation of the evolution problem related to \eqref{expover}. Our construction holds for general Feller process $(M_u)_{u\ge0}$ in $E$ and a general subordinator $(\sigma_u)_{u\ge0}$, but it is in the spirit of Section \ref{sec:harmonic} to study processes on $\R^{d+1}$ and to have a Courrège-type form of $\mathcal{A}$ and $\mathfrak{A}^+$ which we give under mild assumptions on the generator of $(M_u)_{u\ge0}$, in Lemma \ref{lem:subord} and Lemma \ref{lem:operator}.

To this end, let $(M_u)_{u\ge0}$ (a Feller process in $E$) have the transition functions $p_t(x,dy)$, $x\in E, \, t\ge0$, and the corresponding semigroup $(P_t)_{t\ge0}$:
\begin{align}
    P_t f(x)=\int_{E}f(y)p_t(x,dy),\quad x\in E, \, f\in C_0(E),
\end{align}
Denote by $G$ the infinitesimal generator of $(P_t)_{t\ge0}$:
\begin{align}\label{eq:operator-defn}
    Gf(x)\coloneqq\lim_{t\downarrow 0}\frac{P_tf(x)-f(x)}{t}, \quad f\in \DD(G),
\end{align}
where $\DD(G)$ denotes all functions in $C_0(E)$ such that the limit in \eqref{eq:operator-defn} exists in the supremum norm and such that $Gf\in C_0(E)$.

To obtain a part of our results in the case $E=\R^d$, we will sometimes assume the following regularity for $G$:
\begin{assumption}{G}{}\label{assG}
    The process $(M_u)_u$ lives on $\R^d$ and the operator $(G, \mathcal{D}(G))$ has an operator core $C_c^\infty (\mathbb{R}^d)$ for which it holds
    \begin{align}
    \begin{split}
        G\big|_{C_c^\infty (\mathbb{R}^d)}f(x)\, = \, &\sum_{i=1}^d \wt b_i(x) \partial_{x_i} f(x) +\sum_{i,j=1}^d\wt{a}_{ij}(x)\partial^{{2}}_{x_ix_j}f(x)\\& \hspace{-1em}+ \int \left( f(x+y) - f(x)-\sum_{i=1}^d y_i \mathds{1}_{\red{y}\in [-1,1]^{d}} \partial_{x_i} f(x)\right) \wt K(x;dy), 
    \end{split}\label{genM}
\end{align}
\red{where, for each $x\in \R^d$, it holds $\wt b_i(x)\in\R$, $\wt a_{i,j}(x)$ forms a $d\times d$ non-negative definite symmetric matrix, and $\wt K(x,\cdot)$ is a L\'evy measure on $\R^d$.}
\end{assumption}
\red{Assumption~\ref{assG} restricts us to a special class of Feller processes whose generator~$G$ takes the Courr\`ege form; see, e.g., \cite[Theorem~7.2.9]{applebaum-semigroups}. We emphasize that the sole purpose of this assumption is to obtain the generator representation~\eqref{genctrwlim} in a setting more general than that of Example~\ref{examplecoupled}. Accordingly, we do not include a killing term in the generator form~\eqref{genM}, even though such a term appears in the general statement of the Courr\`ege theorem.
}

Let  $(\sigma_u,u\ge 0)$ be a subordinator, i.e. a non-decreasing L\'evy process in $\R$, independent of $M$, characterized by its Laplace exponent
\begin{align}
    \phi(\lambda)=b\lambda+\int_{0}^{+\infty}(1-e^{-\lambda t})\nu(dt),\quad \lambda\ge 0.
    \label{reprphi}
\end{align}
In other words, $\mathds{E}^{(x,v)}\left[e^{-\lambda(\sigma_u-\sigma_0)}\right]=e^{-u\phi(\lambda)}$, for $u,\lambda\ge0$. Here the expectation $\ex^{(x,v)}$ corresponds to the canonical probability measure under which $\pr^{(x,v)}(M_0=x,\sigma_0=v)=1$.  The coefficient $b\ge0$ is called the drift, while the measure $\nu$ satisfies $\int_0^{+\infty}(1\wedge t)\nu(dt)<{+\infty}$, and is called the L\'evy measure of the subordinator. The subordinators are completely described by the exponent $\phi$, and the functions of the form \eqref{reprphi} are called Bernstein functions. For details, we refer to \cite{bernstein}. Sometimes it is useful to use the following representation of $\phi$ which follows from Fubini's theorem:
\begin{align}
    \phi(\lambda)  = \,b \lambda +\lambda\int_0^{+\infty} e^{-\lambda t} \bar{\nu}(t)\, dt,
    \label{reprphibar}
\end{align}
with
\begin{align}
    \bar{\nu}(t) \coloneqq \nu(t, +\infty),\quad t\ge0,
\end{align}
that will be used later.

Let us denote by $\sigma^0=(\sigma_u-\sigma_0)_{u\ge0}$ which is also considered a subordinator, but always started at 0. Note that under $\pr^{(x,0)}$ we have $\sigma^0_u=\sigma_u$ for all $u\geq0$.

In the spirit of Section \ref{sec:harmonic}, in the case when $E=\R^d$ and under the assumption \ref{assG}, let us take a generator form \eqref{genctrwlim} with the following coefficients
\begin{align}
\begin{split}\label{ourgen-0}
    &\red{K(x,t,A_1, A_2)}=K(x,A_1, A_2) \\
 &\qquad\coloneqq \int_{\mathbb{R}^d} \int_{\mathbb{R}} \mathds{1}_{A_1}(y) \mathds{1}_{A_2} (u) p_u(x,x+dy) \nu(du)+b\,\delta_0(A_2)\wt{K}(x,A_1),\\
 &\gamma (x,t) \coloneqq b, \quad a_{ij}(x,t) \coloneqq b\,\wt{a}_{i,j}(x) , \\
 &b_i(x,t) = b_i(x)\coloneqq\int_0^{1}\int_{|y|\leq 1}  y_i \; p_u(x,x+dy)\nu(du)+b\,\wt b_i(x).
\end{split}
\end{align}
\red{Here, in the first line we use a small abuse of notation to emphasize that the jump kernel $K(x,t,A_1,A_2)$ does not depend on the variable $t$.}

In the next lemma, we show that such coefficients in \eqref{ourgen-0} lead to a Feller process $(A_u,\sigma_u)_{u\ge0}$ studied in Section \ref{sec:harmonic} and the process has the representation $(A_u,\sigma_u)=(M_{\sigma_u-\sigma_0}, \sigma_u)$.

\begin{lemma}
\label{lem:subord}
The process $(M_{\sigma_u-\sigma_0}, \sigma_u)_{u\ge 0}$ is a Feller process on $E\times \R$. Under additional assumption \ref{assG}, its infinitesimal generator has an operator core $C_c^\infty(\R^{d})\otimes C_c^\infty(\R)\coloneqq \{(x,t)\mapsto f(x)g(t):f\in C_c^\infty(\R^{d}),\, g\in C_c^\infty(\R)\} $ and on the core the generator is given by
\eqref{genctrwlim} with the coefficients and the jump kernel as in \eqref{ourgen-0}.
\end{lemma}
\begin{proof}
    Note that the process $(M_{\sigma_u-\sigma_0}, \sigma_u)$ is simply the Feller process $(M_u, \Gamma_u)$, $u \geq 0$, where $\Gamma_u$ is a pure drift process\footnote{In other words, $\pr^{(x,t)}(M_u\in dy, \Gamma_u\in ds)=p_u(x,dy)\delta_{t+u}(ds)$, and a comment on its Feller property is in Lemma \ref{app:core}.}, \red{subordinated by} $\sigma_u^0$. Then we can apply \cite[Proposition 15.1]{bernstein} to conclude that $(M_{\sigma_u-\sigma_0}, \sigma_u)_{u\ge0}$ is a Feller process.

    To obtain the representation of the generator of $(M_{\sigma_u-\sigma_0}, \sigma_u)$, we will use  Phillips' Theorem \cite[Theorem 15.6]{bernstein}. Since $(\partial_t,C^1_0(\R))$\footnote{The space $C^1_0(\R)$ contains all $f\in C^1(\R)\cap C_0(\R)$ such that $f'\in C_0( \R) $.} is the infinitesimal generator of $(\Gamma_u)_{u\ge0}$, we denote the infinitesimal generator of $(M_u, \Gamma_u)_{u\ge0}$ by $(G+\partial_t)$ for which it holds 
    \begin{align}
        (G+\partial_t)h(x,t)=Gh(x,t)+\partial_th(x,t),
    \end{align}
    for $h\in \textrm{span}\{(x,t)\mapsto f(x) g(t):f\in \DD(G),g\in C_0^1(\R)\}\subset \DD(G+\partial_t)$.

    Phillips' theorem implies that the infinitesimal generator of $(M_{\sigma_u-\sigma_0}, \sigma_u)$, denoted here by $\mathcal{A}$ has a core $\DD(G+\partial_t)$, and it holds that
    \begin{align}
      \mathcal{A} h(x,t) = b(G+\partial_t)h(x,t)+  \int_0^{+\infty} (\mathds{E}^xh(M_w, t+w) -h(x,t)) \nu(dw),
        \label{1640}
    \end{align}
    for all functions $h\in \DD(G+\partial_t)$.

    Under the additional assumption \ref{assG}, we are in the case $E=\R^d$, and it follows that $C_c^\infty(\R^{d})\otimes  C_c^\infty(\R)$ is an operator core for $(G+\partial_t)$, and by subordination also for $\mathcal{A}$, see Lemma \ref{app:core}. Hence, by regularizing and rearranging the terms in \eqref{1640}, we get that $\mathcal{A}\big|_{C_c^\infty(\R^{d})\otimes  C_c^\infty(\R)}$ has the representation \eqref{genctrwlim} with coefficients as in \eqref{ourgen-0}.
\end{proof}

Before obtaining the operator $\mathfrak{A}^+$, note that the process $L= (L_t)_{t \geq 0}$ defined in \eqref{inverse} becomes under $\pr^{(x,0)}$ the inverse of the subordinator $\sigma^0 = \sigma$, and that the evolution that we study in this case is $A_{L_t}=M_{\sigma_{L_t}}=M_{D_t}$, under $\pr^{(x,0)}$. The process  $(D_t)_{t\ge0}$ is the overshooting of $\sigma$, and it is a right-continuous process which is not a time-homogeneous Markov process\footnote{More on inverse subordinators and overshootings \red{can be found} e.g. in \cite{tlms, bertoin1996, bertoin1999}.}. Since $L_t$ is non-Markov, the same should be expected for $(M_{D_t})_{t\ge0}=(M_{\sigma_{L_t}})_{t\ge0}$. In fact, $(M_{D_t})_{t\ge0}$ is semi-Markov in the sense of Gihman and Skorohod: a random process $Y= (Y_t)_{t \geq 0}$ is said to be semi-Markov in this sense (see \cite[Chapter III]{harlamov}) if the process $(Y, \gamma) = (Y_t, \gamma_t)_{t \geq 0}$, where $\gamma(t) \coloneqq t-0\vee \sup \{ s \geq 0 : Y(s) \neq Y(t) \}$, is a strong Markov process. This can be seen as a consequence of \cite[Theorem 3.1]{Meerschaert2014} from which it follows that $(M_{D_t}, D_t-t)_{t \geq 0}$ is a strong Markov (Hunt) process.

From the point of view of trajectories, since $D$ has jumps and intervals of constancy, the process $Y = (M_{D_t})_{t \geq 0}$ has jumps and intervals of constancy, too. Intuitively, the governing equation of $Y$ should be the operator that is non-local jointly in space-time and coupled, i.e., it acts jointly on the space and time variables ($x$ and $t$) and it cannot be decomposed in two operators acting separately on $x$ and $t$. This can be easily seen by observing that the lengths of the jumps of $D_t$ are determined by the length of its intervals of constancy: this induces dependence between jumps and intervals of constancy in the time-changed process $M_{D_t}$; hence one should expect jointly non-local operator.

We provide now a representation for the operator $\mathfrak{A}^+$ appearing in \eqref{1407-a} that drives the evolution of $(M_{D_t})_{t\ge0}$. Note that $\mathfrak{A}^+$ is obtained from the process $(A_u,\sigma_u)=(M_{\sigma_u-\sigma_0}, \sigma_u)$ by reversing the second coordinate, letting it start from some positive starting point, and stopping it after it crosses 0, see \eqref{transitionunder-NO4}. In other words, in our case the $\mathfrak{A}^+$ is the generator 
of the process $(Z_u^+)_{u\ge0}$ with transitions given by 
\begin{align}\label{transitions-over}
    p^+_u(x,t,dy,dw)=\pr^{(x,t)}(Z_u^+\in(dy,dw))=\pr^{(x,0)}(M_{\sigma_u}\in dy,(t-\sigma_u)\vee 0\in dw).
\end{align}
Such modification in the second coordinate involves the negative translation (pure drift) stopped at zero, i.e. it involves the deterministic process $(\Gamma_u^0)_{u\ge0}$ given by the deterministic transition of $\Gamma^0_u$ from $t$ to $(t-u)\vee 0$. This is a Feller process with a generator denoted by $\partial_t^{(0,{+\infty})}$, \red{with a domain 
\begin{align*}
  C_{00}^1([0,{+\infty}))\coloneqq \{f\in C_0([0,{+\infty})):f\in C^1(0,{+\infty}),\, \lim\limits_{t\searrow0}f'(t)=\lim\limits_{t\to{+\infty}}f'(t)=0\},  
\end{align*}
where $\partial_t^{(0,{+\infty})}f(t)=f'(t)\1_{t>0}$, $f\in C_{00}^1([0,{+\infty}))$.
}

\begin{lemma}\label{lem:operator}
The process $(Z_u^+)_{u\ge0}$ is a Feller process on $E\times[0,{+\infty})$ with the generator $\mathfrak{A}^+$ such that
\begin{align}\label{gen-final-1}
     \mathfrak{A}^+ h(x,t)=b(G-\partial_t^{(0,{+\infty})})h(x,t)+\int_0^{+\infty} \left(P_sh\big(\cdot,(t-s)\vee0\big)(x)-h(x,t)\right)\nu(ds),
\end{align}
for all $h\in \DD(G)\otimes  C_{00}^1([0,{+\infty}))\coloneqq \{(x,t)\mapsto f(x)g(t):f\in \DD(G),\,g\in C_{00}^1([0,{+\infty}))\}$. If \ref{assG} holds, then $\mathfrak{A}^+$ has $C_c^\infty(\R^d)\otimes  C_{00}^1([0,\infty))$ as an operator core.
\end{lemma}
\begin{proof}
    By \eqref{transitions-over}, the process $(Z_u^+)_{u\ge0}$ can be obtained using the subordination with $\sigma^0$ of the Feller process $(M_u,\Gamma_u^0)_{u\ge0}$. In other words, first take the process $(M_u,\Gamma_u^0)_{u\ge0}$ in $E\times[0,{+\infty})$ given by 
    \begin{align}\label{1747}
        \pr^{(x,t)}(M_u\in dy,\Gamma^0_u\in dw)=p_t(x,dy)\delta_{(t-u)\vee 0}(dw).
    \end{align}
    It is easy to show that $(M_u,\Gamma_u^0)_{u\ge0}$ is a Feller process on $E\times[0,{+\infty})$ and that for its generator, denoted by $(G-\partial_t^{(0,{+\infty})})$, it holds that
    \begin{align}
        (G-\partial_t^{(0,{+\infty})})h(x,t)=Gh(x,t)-\partial_t^{(0,{+\infty})}h(x,t),
    \end{align}
    for $h\in \DD(G)\otimes  C_{00}^1([0,{+\infty}))\subset\DD(G-\partial_t^{(0,{+\infty})})$.
    By subordinating the process $(M_u,\Gamma_u^0)_{u\ge0}$ by $\sigma^0$, we get the process with transitions
    \begin{align}
        \pr^{(x,t)}(M_{\sigma^{0}_u}\in dy,\Gamma^0_{\sigma^0_u}\in dw)&=\pr^{(x,0)}(M_{\sigma_u}\in dy,(t-\sigma_u)\vee 0\in dw)\\
        &=\pr^{(x,t)}(Z_u^+\in(dy,dw)),
    \end{align}
    where the first equality comes from \eqref{1747} and subordination, and the second one from \eqref{transitions-over}. 
    Thus, by \cite[Proposition 15.1]{bernstein} we get that $(Z_u^+)_{u\ge0}$ is a Feller process and by Phillips' theorem \cite[Theorem 15.6]{bernstein}, we get that for the generator $\mathfrak{A}^+$ of $(Z_u^+)_{u\ge0}$ it holds that
    \begin{align*}
        \mathfrak{A}^+ h(x,t)&=b(G-\partial_t^{(0,{+\infty})})h(x,t)+\int_0^{+\infty} \left(\ex^{(x,0)}\left[h(M_s,(t-s)\vee0)\right]-h(x,t)\right)\nu(ds)\\
        &=b(G-\partial_t^{(0,{+\infty})})h(x,t)+\int_0^{+\infty} \left(P_sh\big(\cdot,(t-s)\vee0\big)(x)-h(x,t)\right)\nu(ds),
    \end{align*}
    for all $h\in \DD(G-\partial_t^{(0,{+\infty})})$. If \ref{assG} holds, then $\mathfrak{A}^+$ has $C_c^\infty(\R^d)\otimes  C_{00}^1([0,{+\infty}))$ as an operator core by similar calculations as in Lemma \ref{app:core}.
\end{proof}

In the following section, we prove that $q(x,t) = \mathds{E}^{(x,0)}u(M_{D_t})$, $(x,t)\in E\times[0,{+\infty})$ solves the evolution equation \eqref{1407-a} with the operator $\mathfrak{A}^+$ as in Lemma \ref{lem:operator} in the pointwise sense. To be able to do all the calculations, we will not need assumption \ref{assG} but we will need the assumption on the subordinator $\sigma^0$, i.e. on its Laplace exponent.

\begin{assumption}{S}{}\label{assphi}
    The function $\phi$ is a special \red{Bernstein} function such that \red{$b>0$ or} $\nu(0,{+\infty})={+\infty}$, and such that its \red{non-increasing} potential density $u^\phi$ satisfies $\int_0^1|\log(t)|u^\phi(t)dt<{+\infty}$.
\end{assumption}
Let us explain what are the terms and the consequences of the assumption \ref{assphi}. The assumption \red{$b>0$ or $\nu(0,{+\infty})={+\infty}$ means that the subordinator $\sigma^0$ is strictly increasing almost surely or, in other words, $\sigma^0$ is not a compound Poisson process}. That the Bernstein function $\phi$ is special means that the potential measure of the subordinator $\sigma^0$, denoted by $U(A)=\ex[\int_0^{+\infty}\1_A(\sigma^0_s)ds]$, $A\in \BB([0,{+\infty}))$, has a non-increasing density \red{$u^\phi$ on $(0,{+\infty})$}, see \cite[Theorem 13.3]{bernstein}. \red{Moreover, $U$ does not have an atom at 0 since $b>0$ or $\nu(0,{+\infty})={+\infty}$}, see \cite[\red{Eq. (13.4)}]{bernstein}, i.e. $U(dt)= u^\phi(t)dt$. \red{It is important to note that by \cite[Theorem 13.3]{bernstein}, $u^\phi$ is bounded in $[1,+\infty)$ and $\int_0^1u^\phi(t)dt<+\infty$. In the case, $b>0$, $u^\phi$ is also continuous in $(0,+\infty)$ with a finite limit at $0$, see \cite[Theorem 5 in Chapter III]{bertoin1996}.}

Further,  \red{since $\sigma$ is strictly increasing}, $L$ has continuous trajectories and it holds $\mathds{P}^{(x,0)}(L_t>u)=\pr^{(x,0)}(\sigma_u<t)$, while the distribution of $D$ is given by
\begin{align}\label{eq:density-D}
    \mathds{P}^{(x,0)}(D_t\in ds)=\1_{[s>t]}\int_0^t\nu(ds-h)\red{u^\phi(h)dh+bu^\phi(t)\delta_t(ds),}
\end{align}
see \cite[Proposition 2 \red{and Theorem 5}, in Chapter III]{bertoin1996}. 

\red{More on technical aspects of Assumption \ref{assphi} and concrete examples that satisfy \ref{assphi} one can find in Remark \ref{r:assphi}.}

\section{The governing equation of the overshooted process}
\label{sectionmain}
In this section, we make the problem \eqref{1407} explicit and rigorous for the overshooted process introduced in the previous section.
The operator $\mathfrak{A}^+$ has, in this case, the representation \eqref{gen-final-1} that we will consider pointwise (\red{in the sense of Definition \ref{def}}), and does not depend on $t>0$. {In other words, we} study the following harmonic problem
\begin{align}
\begin{cases}
    \mathfrak{A}^+ q(x,t) = 0, \qquad &(x,t) \in E \times (0, +\infty), \\
    q(x,t) =u(x), & (x,t) \in E \times \{ 0 \},
    \end{cases}
    \label{eq:problem}
\end{align}
where
\begin{align}\begin{split}\label{gen-final}
    \mathfrak{A}^+q(x,t) \, = \,&\red{b(G-\partial_t)q(x,t)}\\
    &+ \int_0^{+\infty} (P_sq(x,t-s) \mathds{1}_{[s<t]}+ \mathds{1}_{[s \geq t]}P_sq(x,0) -q(x,t) ) \nu (ds).
\end{split}
\end{align}
\red{
\begin{definition}\label{def}
    We say that $q:E\times[0,+\infty)\to\mathbb R$ is a pointwise solution to \eqref{eq:problem} if the following conditions hold:
    \begin{itemize}
        \item the integral in \eqref{gen-final} converges as a Lebegue integral for every $x\in E$ and $t>0$;
        \item if $b>0$, then $q(\cdot,t)\in\DD(G)$ for every $t>0$, and $q(x,\cdot)\in C^1(0,{+\infty})$ for every $x\in E$;
        \item the equalities in \eqref{eq:problem} hold pointwise.
    \end{itemize}
\end{definition}
}

\begin{remark}
In the Bochner subordination theory, the operator $\mathfrak{A}^+$ is usually denoted by $-\phi(-(\partial_t^{(0,{+\infty})}-G))$, see \cite{SchAJM} for a rigorous argument. In a slightly different setting, in which the time derivative $\partial_t^{(0,{+\infty})}$ was not stopped, fractional powers of the classical heat equation, i.e. $(\partial_t- \Delta)^s$ for $s\in (0,1)$, were studied in \cite{ACM18,FdP24,NS16, ST17}. However, such an operator leads to the functional initial condition (that depends on the whole history $(-\infty,0]$), see also \cite{CS14} for a more general setting.

\end{remark}
Here is our main result.
\begin{theorem}
 \label{thm:operator}
    Let $u\in \DD(G)$ and assume \ref{assphi}. Then a \red{pointwise} solution to \eqref{eq:problem} is given by $q(x,t)=\ex^{(x,0)}[u(M_{D_t})]$, $t\ge0$. Furthermore, this is the unique \red{pointwise} solution such that $q(x,t)$ is  jointly continuous in $E\times [0,+\infty)$ and $\lim_{\red{x}\to +\infty} q(x,t)=0$, \red{locally uniformly in $t\ge0$}.\footnote{\red{The limit here denotes vanishing at infinity locally uniformly in $t\ge0$, i.e., for all $T>0$ and $\varepsilon>0$, there is a compact $K\subset E$ such that $|q(x,t)|\le \varepsilon$, $(x,t)\in \big(E\setminus K\big)\times [0,T]$.}}
\end{theorem}
The proof of this theorem is given in Section \ref{proofmain}, where it is divided into several steps.
For the proof, the regularity of the involved quantities will be crucial (proved in the next lemma) as well as the Laplace transform techniques. Here, for a function $t\mapsto f(t)$, $t \in [0, +\infty)$, with values in some Banach space, we will denote its Laplace transform by
\begin{align}
    \widetilde{f}(\lambda) \coloneqq \int_0^{+\infty} f(t) \, e^{-\lambda t} \, dt \coloneqq \lim_{\tau \to +\infty} \int_0^\tau f(t) \, e^{-\lambda t} \, dt,
    \label{1210}
\end{align}
where the integral should be understood in the Bochner sense (i.e. it is the Lebesgue integral for real-valued functions). We suggest the monograph \cite{abhn} for the main properties of the Laplace transform.
\begin{lemma}\label{l:aux-G-q(t)}
	Under the assumptions of Theorem \ref{thm:operator} the following claims hold:
	\begin{enumerate}[label=(\alph*)]
		\item\label{l-item-1} For all $t \geq 0$, it is true that $q(\cdot,t)\in \DD(G)$,  $Gq(x,t)=\ex^{(x,0)}[Gu(\red{M_{D_t}})]$ for all $x \in E$, and
		\begin{align}
			\|G q(\cdot,t)\|_\infty\le \|Gu\|_\infty,\quad t>0.
		\end{align}
        \red{Moreover, $\lim_{x\to +\infty} q(x,t)=0$, locally uniformly in $t\ge0$.}
		\item\label{l-item-2}
		For all $\lambda >0$, $\wt q(\cdot,\lambda)\in \DD(G)$, and  for all $x \in E$ it is true that $G \wt q(x,\lambda)=\wt{Gq(x ,\lambda)}$.
		\item\label{l-item-3} The function $[0, +\infty) \ni t\mapsto q(x,t)$ is absolutely continuous for all $x \in E$ and there exists $C=C(\phi)>0$ such that, for almost every $t>0$, 
		\begin{align}
			\|\partial_tq(\cdot,t)\|_\infty \le C\big( \|u\|_\infty+\|Gu\|_\infty\big)u^\phi(t),
		\end{align}
		 where $u^\phi(t)$ denotes the potential density of $\sigma$.		
	\end{enumerate}
\end{lemma}
The preceding lemma is proved in Appendix \ref{auxiliary}.
\subsection{Proof of Theorem \ref{thm:operator}}
\label{proofmain}

The proof of the theorem is divided into the following steps.
\begin{enumerate}[label=(\roman*)]
    \item \label{it1} We first show that $\mathfrak{A}^+q(x,t)$ is well defined for all $x \in E$ and for almost every $t\in [0,{+\infty})$, and it admits the Laplace transform.
    \item \label{it2} We then show that
    \begin{align}
        \widetilde{\mathfrak{A}^+q(x,\lambda)} \, = \, 0
    \end{align}
    for all $\lambda >0$ and $x \in E$.
    \item \label{it3} We extend the result obtained in the previous item to our claim about existence by Laplace inversion and continuity arguments.
    \item \label{it4} We prove the uniqueness of the solution.
\end{enumerate}
 For the sake of brevity, in the following calculations, we leave out the space variable $x$ from the calculations wherever unambiguous. This means, for example, that using the notation \eqref{1210}, the Laplace transform of $q(x,t)$ will be denoted by $\lambda\mapsto\wt q(x,\lambda)=\wt q(\lambda)$.

We now prove \ref{it1}. Decompose the operator $\mathfrak{A}^+ q(x,t)$ into its local and  integral parts: $\mathfrak{A}^+ q(x,t)=b(G-\partial_t)q(x,t)+\mathfrak{I}^+ q(x,t)$, and write the integral part as
\begin{align}
    \mathfrak{I}^+ q(x,t) \, \coloneqq \, & \int_{[t, {+\infty})} P_sq(x,0) \nu(ds) + \int_{0}^{+\infty}\big(P_sq(\cdot,t-s)(x)\1_{\{t > s\}}-q(x,t)\big)\nu(ds) \notag\\ = \, &\int_{[t, {+\infty})} P_sq(x,0) \nu(ds)+ \int_{(0,t)}\big(P_sq(\cdot,t)(x)-q(x,t)\big)\nu(ds)\notag\\
    	&+\int_{(0,t)}\big(P_sq(\cdot,t-s)(x)-P_sq(\cdot,t)(x)\big)\nu(ds)-\int_{[t, +\infty)} q(x,t)\nu(ds)\notag\\
    	&\eqqcolon J_0+J_1+J_2+J_3.
        \label{1216}
\end{align}
    First we show that all terms in $\mathfrak{I}^+ q(x,t)$ are well defined for all $(x,t)\in E\times (0,+\infty)$,  and admit the Laplace transform. For $J_0$ it holds
    \begin{align}
       \int_{[t, +\infty)} \left\| P_sq(\cdot,0) \right\|_{\infty} \nu(ds)\, \leq \, \left\| u \right\|_{\infty} \nu[t, +\infty),
    \end{align}
    where we used that $P_s$ is a Feller semigroup and thus $\left\| P_su \right\|_\infty \leq \left\| u \right\|_\infty$. It follows that the integral converges and,  by \eqref{reprphibar}, that it has the Laplace transform for all $\lambda >0$. The Laplace transform is
    \begin{align}\label{eq:1743}
    	\int_0^{+\infty} e^{-\lambda t}\int_{[t, +\infty)} P_su\nu(ds)dt= \, &\int_0^{+\infty} e^{-\lambda t}\int_{(t, +\infty)} P_su \, \nu(ds)dt \notag \\ = \, &\int_0^{+\infty} P_su\left(\frac{1-e^{-\lambda s}}{\lambda}\right)\nu(ds),
    \end{align}
    where we moved from $[t,+\infty)$ to $(t,+\infty)$ using the fact that the L\'evy measure $\nu$ has at most countable number of atoms.

        For $J_1$ we note that, for $u \in \mathcal{D}(G)$, one has that
    \begin{align}\label{eq:J1-bound}
    	|P_sq(\cdot,t)(x)-q(x,t)|\le \min\{2\|u\|_\infty,s\|Gu\|_\infty\}\le c\,(1\wedge s),
    \end{align}
    for all $t>0$, and $x\in E$, where the inequalities come from \cite[Eq. (15.3)]{bernstein} and from Lemma \ref{l:aux-G-q(t)}\ref{l-item-1}. Hence, $J_1$ is well defined, i.e., the integral converges for any $x$, and admits the Laplace transform for all $\lambda>0$.

   For $J_2$, by Lemma \ref{l:aux-G-q(t)}\ref{l-item-3}, we get for all $t>0$ and $x\in E$  that
    \begin{align}
    	|P_sq(\cdot,t-s)(x)-P_sq(\cdot,t)(x)|&\le \left|P_s\left(\int_{t-s}^t\partial_h q(\cdot,h)dh\right)(x)\right|\\
    	&\le C\red{\big(\|u\|_\infty+\|Gu\|_\infty\big)} \int_0^s u^\phi(t-h)dh,\label{eq:J2-bound}
    \end{align}
    and therefore
    \begin{align}
        &\int_0^t |P_sq(\cdot,t-s)(x)-P_sq(\cdot,t)(x)|  \nu(ds)\notag\\
        &\qquad\le C\red{\big(\|u\|_\infty+\|Gu\|_\infty\big)} \int_0^t \int_0^s u^\phi(t-h)\, dh \, \nu(ds) \notag \\
        &\qquad\le C\red{\big(\|u\|_\infty+\|Gu\|_\infty\big)} \int_0^t u^\phi(t-h) \nu(h, +\infty)\red{dh} \notag 
        = \, C\red{\big(\|u\|_\infty+\|Gu\|_\infty\big)},
    \end{align}
    where we used that
    \begin{align}
        \int_0^t u^\phi(t-h) \nu(h, +\infty) \, dh \, \le \, 1,
    \end{align}
    as it follows, for example, from \cite[Theorem 13.9 \& Eq. (13.12)]{bernstein} or \cite[Proposition 2 in Chapter III]{bertoin1996}.
    Therefore $J_2$ is well defined and admits the Laplace transform for all $\lambda>0$. 
    
    For $J_3$ we note that for all $t>0$ and $x\in E$ we  have
    \begin{align}\label{eq:J3-bound}
    	\int_{[t, {+\infty})} |q(x,t)|\nu(ds)&\le \|u\|_\infty 	\red{ \nu[t, +\infty)}
    \end{align}
    so $J_3$ admits the Laplace transform for all $\lambda>0$ using \eqref{reprphibar}, see also the comment below \eqref{eq:1743}. 

    Combining the preceding arguments, we conclude that $\mathfrak{I}^+q(x,t)$ is well defined for every $x\in E$ and $t> 0$, and that its Laplace transform exists for every $\lambda>0$.
    
    For the local part of $\mathfrak A^+$, {by using} Lemma \ref{l:aux-G-q(t)}, we have that $(G-\partial_t)q(x,t)=Gq(x,t)-\partial_tq(x,t)$ is well defined for every $x\in E$ and almost every $t>0$. By the bounds provided by the same Lemma, its Laplace transform is also defined for all $\lambda>0$ and it holds:
    \begin{align}\label{1622} 
        \int_{0}^{+\infty} e^{-\lambda t}(G-\partial_t)q(t)dt=G\wt q(\lambda)-\lambda \wt q(\lambda)+u=(G-\lambda \II)\wt q(\lambda)+u.
    \end{align}

    This proves that $\mathfrak{A}^+ q(x,t)$ is convergent for any $x$ and almost every $t$, and admits the Laplace transform for all $\lambda>0$.
    
Now we prove Item \ref{it2}. It is useful first to compute the Laplace transform of $q(t)$. To this end, using the distribution of the overshooting \eqref{eq:density-D}, we have that
    \begin{align}
            \wt q(\lambda)&=\int_0^{+\infty} e^{-\lambda t}\int_0^t\int_{t-h}^{+\infty} P_{s+h}u\,\nu(ds)u^\phi(h)dh\,dt+b\int_0^{+\infty} e^{-\lambda t} u^\phi(t)P_tu\,dt\nonumber\\
            &=\int_0^{+\infty} \int_0^{+\infty} P_{s+h}u \left(\int_h^{s+h}e^{-\lambda t}dt\right)\nu(ds)u^\phi(h)dh+b\int_0^{+\infty} e^{-\lambda t} u^\phi(t)P_tu\,dt\nonumber\\
        \begin{split}\label{eq:Laplace-q}
            &=\int_0^{+\infty} \int_0^{+\infty} e^{-\lambda h}P_{s+h}u \left(\frac{1-e^{-\lambda s}}{\lambda}\right)\nu(ds)u^\phi(h)dh\\
            &\qquad+b\int_0^{+\infty} e^{-\lambda t} u^\phi(t)P_tu\,dt,
        \end{split}
    \end{align}
    where in the second equality we used Fubini's theorem (by $\|P_t u\|_\infty\le \|u\|_\infty$ and integrability properties of $\nu$ and $u^\phi$).
    
    The Laplace transform of {$\mathfrak{A}^+q(t)$ can be computed as follows}:
    \begin{align}\label{LaplaceU+}
        \int_0^{+\infty} e^{-\lambda t} \mathfrak{A}^+q(t)dt=\int_0^{+\infty} e^{-\lambda t} b(G-\partial_t)q(t)dt+\int_0^{+\infty} e^{-\lambda t} \mathfrak{I}^+q(t)dt.
    \end{align}
    For the first integral, by \eqref{1622} and \eqref{eq:Laplace-q}, we have
    \begin{align}
        \int_0^{+\infty} e^{-\lambda t} b(G-\partial_t)q(t)dt&=bu+b\int_0^{+\infty} e^{-\lambda t} u^\phi(t)P_t(G-\lambda\II)u\,dt\label{2217-0}\\
            &\hspace{-6em}+b\int_0^{+\infty} \int_0^{+\infty} e^{-\lambda h}P_{s+h}(G-\lambda\II)u \left(\frac{1-e^{-\lambda s}}{\lambda}\right)\nu(ds)u^\phi(h)dh,\label{2217}
    \end{align}
    where we moved the operator $G$ inside the integral since it is a closed operator \cite[Proposition 1.6.3]{abhn} (or, by a direct argument as in the proof of Lemma \ref{l:aux-G-q(t)}), and where we also used that $G$ and $P_t$ commute on $\DD(G)$.

Now we compute the Laplace transform of $\mathfrak{I}^+ q(x,t)$. Using Fubini's theorem, we get that
    \begin{align}
        &\int_0^{+\infty} e^{-\lambda t}\mathfrak{I}^+ q(t)dt \notag \\ = \, & \int_0^{+\infty} e^{-\lambda t}\int_{[t, +\infty)} P_su\nu(ds)dt +\int_0^{+\infty} e^{-\lambda t}\int_{0}^{+\infty}\big(P_sq(t-s)\1_{\{t> s\}}-q(t)\big)\nu(ds)dt \notag \\
        = \, & \int_0^{+\infty} P_s u \left( \frac{1-e^{-\lambda s}}{\lambda} \right) \nu(ds) + \int_0^{+\infty} e^{-\lambda t}\int_{0}^{+\infty}\big(P_sq(t-s)\1_{\{t> s\}}-q(t)\big)\nu(ds)dt,
        \label{1212}
        \end{align}
where we used \eqref{eq:1743}. The second term in \eqref{1212} is
            \begin{align}
      & \int_0^{+\infty} e^{-\lambda t}\int_{0}^{+\infty}\big(P_sq(t-s)\1_{\{t> s\}}-q(t)\big)\nu(ds)dt  \notag \\ = \, &\int_0^{+\infty} \int_{0}^{+\infty} e^{-\lambda t}\big(P_sq(t-s)\1_{\{t\ge s\}}-q(t)\big)dt\,\nu(ds)\notag \\
        = \, &\int_0^{+\infty} \big(e^{-\lambda s}P_s\wt q(\lambda)-\wt q(\lambda)\big)\nu(ds)\label{eq:1619}.
    \end{align}
    Note that $(e^{-\lambda s}P_s)_{s \geq 0}$ is a semigroup with the infinitesimal generator $G-\lambda \II$, with the domain $\DD(G)$. Hence, 
    \begin{align}
       &  \int_0^{+\infty} e^{-\lambda t}{\mathfrak{I}^+} q(t)dt \notag \\  = \, &\int_0^{+\infty} P_s u \left( \frac{1-e^{-\lambda s}}{\lambda} \right) \nu(ds) +\int_0^{+\infty} \int_0^s e^{-\lambda t}P_t(G-\lambda \II)\wt q(\lambda)dt\,\nu(ds)\\
        = \, & \int_0^{+\infty} P_s u \left( \frac{1-e^{-\lambda s}}{\lambda} \right) \nu(ds) + \int_0^{+\infty} e^{-\lambda t}P_t(G-\lambda \II)\wt q(\lambda)\overline\nu(t)dt.\label{eq:1529}
    \end{align}
    Here, in the third equality we used Fubini's theorem which is justified since $u\in \DD(G)$ by the following argument. Indeed, by Lemma \ref{l:aux-G-q(t)}\ref{l-item-2} $\wt q(\lambda)\in \DD(G)$, so $P_t\wt q(\lambda)\in \DD(G)$. Further, $P_t$ and $G$ commute on $\DD(G)$ which gives us the bound
    \begin{align*}
    	|P_t(G-\lambda \II)\wt q(\lambda)(x)|\le \|(G-\lambda \II)\wt q(\lambda)\|_\infty\le \|G\wt q(\lambda)\|_\infty+\lambda\|\wt q(\lambda)\|_\infty<{+\infty}.
    \end{align*}
{Now we implement the expression for $\wt q$ from \eqref{eq:Laplace-q} into \eqref{eq:1529}, and prove that \eqref{LaplaceU+} gets annihilated. We do it term by term.} By substituting {the first term of} \eqref{eq:Laplace-q} into the second term in \eqref{eq:1529}, we obtain
    \begin{align}
        &\int_0^{+\infty}\int_0^{+\infty} \int_0^{+\infty} e^{-\lambda t}P_t(G-\lambda \II)e^{-\lambda h}P_{s+h}u \left(\frac{1-e^{-\lambda s}}{\lambda}\right)\nu(ds)u^\phi(h)dh \overline\nu(t)dt\\
        &=\int_0^{+\infty} P_s\left(\int_0^{+\infty} \int_0^{+\infty} e^{-\lambda (h+t)}P_{t+h}(G-\lambda \II)u\, u^\phi(h)dh \overline\nu(t)dt\right)\left(\frac{1-e^{-\lambda s}}{\lambda}\right)\nu(ds).\label{eq:1536}
    \end{align}
    The inner integral in \eqref{eq:1536} is equal to
    \begin{align}
        \int_0^{+\infty} \int_0^{+\infty} &e^{-\lambda (h+t)}P_{t+h}(G-\lambda \II)u\, u^\phi(h)dh \overline\nu(t)dt\\
        &=\int_0^{+\infty} \int_h^{+\infty} (G-\lambda \II)e^{-\lambda v}P_{v}u\,  \overline\nu(v-h)dv\,u^\phi(h)dh\\
        &=\int_0^{+\infty} e^{-\lambda v}P_{v}(G-\lambda \II)u\,  \int_0^v\overline\nu(v-h)u^\phi(h)dhdv\\
        &=\int_0^{+\infty} e^{-\lambda v}P_{v}(G-\lambda \II)u\,{\big(1-bu^\phi(v)\big)}\,dv\\
        &=-u{-b\int_0^{+\infty} e^{-\lambda v}P_{v}(G-\lambda \II)u\,u^\phi(v)dv}.\label{eq:1545}
    \end{align}
    Here, in the second line we used the change of variables, in the third Fubini's theorem, in the fourth the fact that $\overline\nu(v-h)u^\phi(h)dh{+bu^\phi(v)\delta_v(dh)}$ is the distribution of $\sigma_{L_t-}$ (see \cite[Proposition 2 {\& Theorem 5} in Chapter III]{bertoin1996}), and in the final line the resolvent property $(\lambda \II-G)^{-1}=\int_0^{+\infty} e^{-\lambda v}P_{v}dv$.
    By substituting \eqref{eq:1545} into \eqref{eq:1536} and then in \eqref{eq:1529}, we {can see that the first term of \eqref{1212} and the term in \eqref{2217} are annihilated. To prove that \eqref{LaplaceU+} equals to 0, it remains to collect the remaining terms. These are the terms in \eqref{2217-0} and the second term of \eqref{eq:Laplace-q} implemented into the second term of \eqref{eq:1529}. That is, we need to prove that
    \begin{align}
    \begin{split}\label{2226}
        bu+&b\int_0^{+\infty} e^{-\lambda t} u^\phi(t)P_t(G-\lambda\II)u\,dt\\&+b\int_0^{+\infty}\int_0^{+\infty} e^{-\lambda (t+v)}\overline{\nu}(t)u^\phi(v)P_{t+v}(G-\lambda\II)u\, dv\,dt=0.
    \end{split}
    \end{align}
    To this end, define bounded operators on $C_0(E)$ by
\begin{align}
\mathcal U_\lambda^\phi h
&\coloneqq\int_0^{+\infty} e^{-\lambda v}P_vh\,u^\phi(v)\,dv,\\
\mathcal K_\lambda^\phi h
&\coloneqq bh+\int_0^{+\infty} e^{-\lambda t}\overline\nu(t)P_th\,dt,    
\end{align}
so that \eqref{2226} can be written as
\begin{align}\label{2237}
    bu+\mathcal K_\lambda^\phi\left[b\mathcal U_\lambda^\phi(G-\lambda\II)u\right]=0.
\end{align}
Note that $(G-\lambda\II)$ commutes with the operators $\mathcal U_\lambda^\phi$ and $\mathcal K_\lambda^\phi$, since $G$ is closed and commutes with $P_t$. Furthermore, 
For $h\in C_0(E)$, the semigroup property and Fubini's theorem imply
\begin{align*}
\mathcal K_\lambda^\phi\mathcal U_\lambda^\phi h
&=b\int_0^{+\infty} e^{-\lambda v}P_vh\,u^\phi(v)\,dv+\int_0^{+\infty}\int_0^{+\infty}
 e^{-\lambda(t+v)}\overline\nu(t)P_{t+v}h\,u^\phi(v)\,dv\,dt\\
&=\int_0^{+\infty} e^{-\lambda r}P_rh
\left(bu^\phi(r)+\int_0^r\overline\nu(t)u^\phi(r-t)\,dt\right)dr\\
&=\int_0^{+\infty} e^{-\lambda r}P_rh\,dr=(\lambda \II-G)^{-1}h,
\end{align*}
where the second equality holds since  $\overline\nu(v-h)u^\phi(h)dh+bu^\phi(v)\delta_v(dh)$ is the distribution of $\sigma_{L_t-}$. By substituting this into the left-hand side of \eqref{2237}, we obtain that \eqref{2237}, and thus \eqref{2226}, holds true.
}

    In other words, we get that
    \begin{align}
    	\int_0^{+\infty} e^{-\lambda t} \mathfrak{A}^+ q(x,t)dt=0,
        \label{1202}
    \end{align}
    for every $x \in E$. 
    
    Now we prove Item \ref{it3}. It follows from \eqref{1202} that for every $x$ there exists a null set $I_x \subset (0, +\infty)$ such that $\mathfrak{A}^+q(x,t) =0$ for all $t \notin I_x$.
    We proceed to show that $\cup_x I_x = \emptyset$.

    First, note that $q(x,t)=\ex^{(x,0)}[u(M_{D_t})]=\ex[P_{D_t}u(x)]$ so $q(x,t)$ is jointly continuous in $E\times [0,+\infty)$. This can be seen by the dominated convergence theorem since \red{by Feller property} $(s,x)\mapsto P_su(x)$ is jointly continuous, \red{see \cite[Lemma 1.4]{LevyMatters3}}, and $D_t$ is right-continuous such that at fixed time $t\in [0,+\infty)$ the probability of jump is zero \red{since $\sigma$ is not a compound Poisson process}. The same argument shows that $Gq(x,t)=\ex^{(x,0)}[Gu(M_{D_t})]$ is jointly continuous as well. Moreover, a similar argument, again using the Feller property, yields $\lim_{\red{x}\to+\infty}q(x,t)=0$.
    Lastly, recall that, up to this point, we only know that $\partial_t q(x,t)$ exists only for almost every $t>0$.
    
    Let us prove that $I_x=\emptyset$ for each $x\in E$. Assume first $b=0$.
    The decomposition \eqref{1216} and the already obtained bounds \eqref{eq:J1-bound}, \eqref{eq:J2-bound} and \eqref{eq:J3-bound}, allow us to use the dominated convergence theorem to get that  $$(0, +\infty) \ni t \mapsto \mathfrak{A}^+ q(x,t)=\mathfrak{I}^+ q(x,t)=\int_0^{+\infty} \big(P_sq(\cdot,(t-s)\vee 0)(x)-q(x,t)\big)\nu(ds)$$ is continuous in $t$ for any $x \in E$, since all the terms inside the integral are continuous in $t$. This implies that $I_x = \emptyset$ for all $x \in E$.

\red{
    Assume now $b>0$, fix $x\in E$, and note that we can rewrite $\mathfrak{A}^+q(x,t) =0$ as
    \begin{align}
        \partial_t q(x,t)=Gq(x,t)+\frac{1}{b}\mathfrak{I}^+q(x,t),\quad t\notin I_x.
    \end{align}
    Now, since $I_x$ is a null set and by Lemma \ref{l:aux-G-q(t)}\ref{l-item-3}, we may integrate the last relation to obtain
    \begin{align}\label{2302}
        q(x,t)-q(x,0)=\int_0^t\left(Gq(x,s)+\frac{1}{b}\mathfrak{I}^+q(x,s)\right)ds,\quad t>0.
    \end{align}
    However, since $Gq(x,s)=\ex^{(x,0)}[Gu(M_{D_s})]$ is continuous in $s$ and since $\mathfrak{I}^+q(x,s)$ is continuous in $s$ as well as in the previous case, it follows that the right-hand side of \eqref{2302} is differentiable. Therefore, the left-hand side is also differentiable. This implies that, in the case $b>0$, $\partial_t q(x,t)$ exists for all $t>0$ and is continuous in $(0,+\infty)$, and moreover that $t\mapsto\mathfrak A^+q(x,t)$ is continuous. Therefore, $I_x=\emptyset$.

    This finishes the proof that $q(x,t)$ is a pointwise solution to \eqref{eq:problem}.
}

Now we prove Item \ref{it4}, that is, uniqueness. Without loss of generality, in what follows, we may assume that $\nu\not\equiv0$. Indeed, the case $\nu\equiv 0$ reduces the operator $\mathfrak A^+$ to $b(G-\partial_t)$ so the problem \eqref{eq:problem} becomes the abstract Cauchy problem for which uniqueness (under the conditions of the theorem) is well known, see \cite[Theorem 1.3.15]{applebaum-semigroups}.

Let $q_2(x,t)$  be another \red{pointwise} solution to \eqref{eq:problem} which is jointly continuous and vanishing at infinity \red{locally uniformly in $t>0$}. Assume that $\overline q=q-q_2\not\equiv 0$, \red{so, without loss of generality, there exists $T>0$ such that $\sup_{x\in E,\,t\in[0,T]}\overline q(x,t)>0$. Let $(x_\star,t_\star)\in E\times [0,T]$ such that $\overline q(x_\star,t_\star)=\sup_{\{x\in E,\,t\in[0,T]\}}\overline q(x,t)$. Note that $(x_\star,t_\star)$ is well defined by the joint continuity of $\overline q$ and its vanishing as $x\to +\infty$ locally uniformly in $t>0$. Further,} the initial condition implies $\overline q(\cdot,0)\equiv0$, so we have $t_\star>0$. Now we proceed to show that $\mathfrak{A}^+ \overline q(x_\star,t_\star)<0$, which is a contradiction with the assumption that $q_2$ is a solution.


To this end, note that $P_s \overline q(\cdot,t_\star-s)(x)\le \overline q(x_\star,t_\star)$ for $(x,s)\in E\times [0,t_\star]$. \red{Also, if $b>0$, by Definition \ref{def}, we have $\overline q\in \DD(G)$, so $G\overline q(x_\star,t_\star)\le 0$ since $G$ satisfies the classical positive maximum principle \cite[Proposition 7.2.2]{applebaum-semigroups}. Furthermore, $\partial_t\overline q(x_\star,t_\star)\ge0$ since, if $\partial_t\overline q(x_\star,t_\star)<0$, then for some $\varepsilon>0$ it would hold $\overline q(x_\star,t_\star-\varepsilon)>\overline q(x_\star,t_\star)$ contradicting the choice of $(x_\star,t_\star)$.} Thus,
\begin{align*}
    \mathfrak{A}^+ \overline q(x_\star,t_\star)&=\red{b(G-\partial_t)\overline q(x_\star,t_\star)+}\int_0^{t_\star}\left( P_s \overline q(\cdot,t_\star-s)(x_\star)\1_{s<t_\star}-\overline q(x_\star,t_\star)\right)\nu(ds) \\
    &\hspace{4em} + \int_{[t_\star,+\infty)} \left(P_s \overline q(\cdot,0)(x_\star)-\overline q(x_\star,t_\star)\right)\nu(ds)\\
    &\le -\overline q(x_\star,t_\star) \nu[t_\star,+\infty)<0.
\end{align*}
Here, $\nu[t_\star,+\infty)>0$ since by \ref{assphi}, $\phi$ is a special Bernstein function, see \cite[Proposition 13.17]{bernstein}. This finishes the proof of uniqueness. 
\qed

The proof of the uniqueness above is in essence the proof of the positive maximum principle.
\begin{proposition}[Positive maximum principle]
    Let $f:E\times [0,{+\infty})\to \R$ and $(x_\star,t_\star)\in E\times [0,+\infty)$ be such that $f(\cdot,0)\in C(E)$, $f(x_\star,t_\star)\ge f(x,t)$, $(x,t)\in E\times (0,t_\star]$, and $f(x_\star,t_\star)> f(x,0)\red{\ge0}$, $x\in E$, and that $\mathfrak{A}^+f(x_\star,t_\star)$ is well defined pointwise \red{in the sense of Definition \ref{def} and that $\nu\not\equiv0$}. Then
    \begin{align}\label{pmp}
        \mathfrak{A}^+f(x_\star,t_\star)<0.
    \end{align}
\end{proposition}
\begin{proof}
    The only non-trivial part, compared with the final part of the proof of Theorem \ref{thm:operator} is to prove
    \begin{align}\label{1802}
        \int_{[t_\star,+\infty)} \left(P_s f(\cdot,0)(x_\star)-f(x_\star,t_\star)\right)\nu(ds)<0.
    \end{align}
    Assume first that \red{$M$ cannot die when started from $x_\star$, i.e.} $\pr_{x_\star}(M_s\in E)=1$, $s\ge0$. Since $\phi$ is special, there is $[a,b]\subset [t_\star,+\infty)$ such that $\nu[a,b]>0$, see \cite[Proposition 13.17]{bernstein}. Further, it is easy to prove by contradiction that there is $\varepsilon>0$ and a compact $K\subset E$ such that $\pr^{x_\star}(M_s \in K)>\varepsilon$, $s\in [a,b]$, since $(M_u)_u$ is a Feller process. Hence, from the continuity of $f(\cdot,0)$, \red{and the assumption $f(x_\star,t_\star)> f(x,0)$, $x\in E$,} we get that, \red{since $K$ is compact}, there is $\overline \varepsilon>0$ such that $f(x,0)+\overline \varepsilon<f(x_\star,t_\star)$, $x\in K$. Therefore, for $s\in [a,b]$ we get
    \begin{align*}
        P_s f(\cdot,0)(x_\star)-f(x_\star,t_\star)&= \int_{E} \big(f(y,0)-f(x_\star,t_\star)\big)p_s(x_\star,dy)\\
        &=\left(\int_K+\int_{K^c}\right)\big(f(y,0)-f(x_\star,t_\star)\big)p_s(x_\star,dy)< -\varepsilon\overline \varepsilon,
    \end{align*}
    which gives \eqref{1802}.

    If $M$ can die when started from $x_\star$, the calculation is even simpler:
    \begin{align*}
        P_sf(\cdot,0)(x_\star)-f(x_\star,t_\star)&=P_s\left(f(\cdot,0)-f(x_\star,t_\star)\right)(x_\star)\\
        &\hspace{4em}-f(x_\star,t_\star)\pr_{x_\star}(\text{$M$ is dead at time $s$})\\
        &\red{\le -f(x_\star,t_\star)\pr_{x_\star}(\text{$M$ is dead at time $s$})}
    \end{align*}
    \red{and thus
    \begin{align*}
        \int_{[t_\star,+\infty)} &\left(P_s f(\cdot,0)(x_\star)-f(x_\star,t_\star)\right)\nu(ds)\\&\le-f(x_\star,t_\star)\int_{t_\star}^{+\infty}\pr_{x_\star}(\text{$M$ is dead at time $s$})\nu(ds)<0,
    \end{align*}
    }where the strict inequality holds since \red{$f(x_\star,t_\star)>f(x_\star,0)\ge0$, since the probability}  $s\to \pr_{x_\star}(\text{$M$ is dead at time $s$})$ is monotone with a strictly positive limit at infinity, and \cite[Proposition 13.17]{bernstein}.
    \end{proof}

\red{
    \begin{remark}
        We note that if the Feller process $M$ is conservative, that is $\pr_x(M_s\in E)=1$ for all $x\in E$ and $s\ge0$, then the positive maximum principle also holds without the assumption $f(x,0)\ge0$, $x\in E$. Indeed, this assumption was used only in the case in which $M$ can die when started from $x_\star$.

        Also, the assumption $\nu\not\equiv 0$ is needed to obtain the strict inequality in \eqref{pmp}. Indeed, if $\nu\equiv0$, the operator $\mathfrak{A}^+$ becomes $b(G-\partial_t)$ for which it is well known that, in general, only the inequality in \eqref{pmp} holds. 
    \end{remark}

\begin{remark}\label{r:assphi}
    The core of Assumption~\ref{assphi}, as revealed by the proof of Theorem~\ref{thm:operator}, is the existence of a non-increasing potential density $u^\phi$. This property allows us to establish, a priori, that the function $q(x,t)=\ex^{(x,0)}u(M_{D_t})$ is absolutely continuous in $t$, i.e., that $\partial_t q(x,t)$ exists for almost every $t>0$; see Lemma~\ref{l:aux-G-q(t)}\ref{l-item-3}. To obtain this conclusion we also require the log-integrability condition in \ref{assphi}. The existence of $\partial_t q(x,t)$ is crucial for showing that the operator $\mathfrak A^+$ is well defined and extending Theorem~\ref{thm:operator} to more general Bernstein functions $\phi$ would require a different strategy to prove an analogue of Lemma~\ref{l:aux-G-q(t)}\ref{l-item-3}.

    The class of Bernstein functions satisfying \ref{assphi} is large. For an extensive list of complete (hence special) Bernstein functions, see \cite[Chapter~20]{bernstein}. The log-integrability condition $\int_0^1 |\log\,t| u^\phi(t)dt < +\infty$ can be verified using the bound
    \begin{align}\label{u bound}
    u(t)\le (1-2e^{-1})\frac{1}{t\,\phi(t^{-1})}, \quad t>0,
    \end{align}
    see \cite[Proposition~3.3 and Lemma~A.1]{KimMimica2012}; the proof of \eqref{u bound} in \cite{KimMimica2012} relies only on the assumption that $\phi$ is special. From \eqref{u bound} one deduces that, for example, the following subordinators satisfy \ref{assphi}:
    \begin{itemize}
        \item stable subordinator $\phi(\lambda)=\lambda^\alpha$, $\alpha\in(0,1]$;
        \item tempered stable subordinator $\phi(\lambda)=(\lambda+m)^\alpha-m^\alpha$, $\alpha\in(0,1]$;
        \item any $\phi$ that is regularly varying at infinity with index $\alpha\in(0,1)$;
        \item any $\phi$ for which there exist $c>0$, $\lambda_0>0$, and $\alpha\in(0,1)$ such that $\phi(\lambda)\ge c\,\lambda^\alpha$, for all $\lambda\ge\lambda_0$.
    \end{itemize}
    In contrast, the geometric stable subordinator $\phi(\lambda)=\log(1+\lambda^\alpha)$, $\alpha\in(0,1]$, fails the log-integrability condition, since $u^\phi(t)\sim \frac{2}{\alpha\,t(\log t)^2}$, as $t\to0$; see \cite[Theorem~2.2]{SSV06}.
\end{remark}
}
\section{Harmonic approach in action} 
\label{sec:ex}

 \red{As already said in Subsection \ref{ss:harmonic approach algorithm},} in the forthcoming examples, we put our harmonic method for governing equations into action in the cases where the form of the operator and the required regularity are already known, thereby independently recovering the best-understood specific cases. In doing so, our approach indeed demonstrates a probabilistic unifying approach to understanding the governing equations for (O)CTRWs limit processes across all settings.

\subsection{\red{Classical Cauchy problem}}
Let $(M_u)_u$ be a Feller process on a locally compact separable metric space $E$ with the generator $G$ and we want to derive the governing equation for $q(x,t)=\ex^x u(M_t)$. This can be obtained by our generalized harmonic approach: use Lemma \ref{lem:operator} with the subordinator $\sigma_u$ being the pure drift, that is $\sigma_u=u$, which corresponds to $b=1$ and $\nu\equiv 0$. Therefore, since $L_t=t$, by Lemma \ref{lem:operator}, the function $q(x,t)=\ex^x u(M_t)$ is related to the equation
 \begin{align}
    \begin{cases} \bigl(G-\partial_t^{(0,{+\infty})}\bigr)q(x,t) \, = \, 0 \qquad &(x,t) \in E \times (0, +\infty) \\ q(x,t) = u(x) & (x,t) \in E \times \{ 0 \},
\end{cases}
\end{align}
Recall that $\partial_t^{(0,{+\infty})}f(t)=f'(t)\1_{t>0}$, so in other words, we have reobtained the classical  Cauchy problem: the function $q(x,t)=\ex^x u(M_t)$ solves $\partial_tq(x,t)=Gq(x,t)$, $t>0$, $x\in E$.

We note that $\mathfrak{A}^+$ and $\mathfrak{A}^-$ coincide in this context, since $p^+$ and $p^-$ generate  the same processes (because $\sigma_u$ is the pure drift).

In the context of (O)CTRWs, the pure drift $\sigma_u$ can be obtained, for example, by taking i.i.d. waiting times $W_n\sim \textrm{Exp}(1)$, the scaling $W_n^c=c^{-1}W_n$, and employing the law of large numbers.

\subsection{Uncoupled CTRWs and fractional kinetic equations}
In the uncoupled case, i.e., when the waiting times and the jumps are independent, the limit of the CTRW and the OCTRW coincide. This can be seen as follows. If the coordinates of the Feller process $(A_u, \sigma_u)_u$, with generator \eqref{genctrwlim}, are independent then $A_u$ and $\sigma_u$ are independent Feller processes on $\mathbb{R}^d$ and $\mathbb{R}$, respectively. Denote their generator by $(G, \mathcal{D}(G))$ (for the process $A$) and by $(S, \mathcal{D}(S))$ (for the process $\sigma$). Then the generator \eqref{genctrwlim} becomes
\begin{align}
    \mathcal{A}f(x,t) \, = \, Gf(x,t) +Sf(x,t)
\end{align}
where $G$ and $S$ act upon different variables. In particular, we can obtain from \eqref{genctrwlim} the Courr\`ege-type form
\begin{align}
    \mathcal{A}f(x,t) \, = \, &\sum_{i=1}^db_i(x)\partial_{x_i}f(x,t) + \frac{1}{2} \sum_{1 \leq i, j \leq d} a_{ij}(x) \partial_{x_ix_j}^2f(x,t) \notag \\
    & + \int_{\mathbb{R}^d}  \left( f(x+y,t) -f(x,t) - \sum_{i=1}^d y_i \mathds{1}_{y \in [-1,1]^d} \partial_{x_i}f(x,t)  \right) K_G(x,dy)\notag \\
    &+\gamma(t) \partial_t f(x,t) + \int_0^{+\infty} \left( f(x,t+w) -f(x,t) \right) \nu(t,dw)
    \label{1348}
\end{align}
where $K_G$ and $\nu$ are the jump kernels of the two processes $A$ and $\sigma$, respectively. Comparing \eqref{genctrwlim} with \eqref{1348} one can see that the jump kernel of $(A_u, \sigma_u)$ is
\begin{align}
    K(x,t;dy,dw) \, = \, K_G(x, dy) \delta_0(dw) + \delta_0(dy)\nu(t, dw)
\end{align}
which means that $K(x,t;dy,dw)$ is supported on $ \mathbb{R}^d \times \{0 \}  \cup \{ 0 \} \times (0, +\infty)$, i.e., the process $A_u$ and $\sigma_u$ do not jump simultaneously. It follows that $A_{L_t-}=A_{L_t}$, a.s. (see also \cite[Lemma 3.9]{straka2011lagging}).

The prototype of this theory is the case when $(\sigma_u)_{u \geq 0}$, is a strictly increasing subordinator with the Laplace exponent given by some Bernstein function $\phi$, see \eqref{reprphi}, with $\nu(0, +\infty) = +\infty$, i.e., the subordinator is not a compound Poisson process. In this case the operator \eqref{1348} reduces to
\begin{align}
      \mathcal{A}f(x,t) \, = \, &\sum_{i=1}^db_i(x)\partial_{x_i}f(x,t) + \frac{1}{2} \sum_{1 \leq i, j \leq d} a_{ij}(x) \partial_{x_ix_j}^2f(x,t) \notag \\
    & + \int_{\mathbb{R}^d}  \left( f(x+y,t) -f(x,t) - \sum_{i=1}^d y_i \mathds{1}_{y \in [-1,1]^d} \partial_{x_i}f(x,t)  \right) K_G(x,dy)\notag \\
    &+b  \partial_t f(x,t) + \int_0^{+\infty} \left( f(x,t+w) -f(x,t) \right) \nu(dw)
    \label{1753}
\end{align}
where $b \geq 0$ and $\nu(\cdot)$ come from the representation of $\phi$ in \eqref{reprphi}. 

The transition probabilities $p^+$ and $p^-$ defined in \eqref{transitionover} and \eqref{transitionunder} become the same, since $A_{L_t}=A_{L_t-}$, a.s. Therefore, in order to study the evolution of the process $A_{L_t}$, $t \geq 0$, it is convenient in this case to look at the transition probabilities \eqref{transitionover-NO2} only. We get
\begin{align}
    p_u^+(x,t;dy,dw) \coloneqq  \mathds{P}_A^{x} \left( A_u \in dy \right) \mathds{P}_\sigma\big( (t- \sigma_u) \vee 0 \in dw  \big) 
    \label{57}
\end{align}
where $\mathds{P}_A$ and $\mathds{P_\sigma}$ denote the canonical measures on $A$ and $\sigma$.
The generator of the process with transition probabilities \eqref{57} is, therefore
\begin{align}
   \mathfrak{A} q(x,t) \, = \, &  Gq(x,t)\notag \\
    &-b \partial_t q(x,t) - \int_0^{+\infty} \left( q(x,t-w) \mathds{1}_{[w < t]} + q(x,0) \mathds{1}_{[w\geq t]} -q(x,t) \right) \nu(dw).
\end{align}
The corresponding harmonic problem  \eqref{1407-a} reduces, to the non-local in time fractional kinetic case studied, for example, in \cite{kolokoltsov2009generalized} and whose solution is, indeed $q(x,t) = \mathds{E}^x u(A_{L_t})$, for $u \in \text{Dom}(G)$. Rearranging the operator acting on the time variable $t>0$, one gets different forms of fractional-type operators that have been studied, for example, in \cite{chen, hernandez2017, kochubei}.

\subsection{Space-dependent CTRWs and variable-order fractional kinetic equations}
It is possible that the waiting times of a CTRW depends on the position of the walker. Such construction is made explicit, for example, in \cite{Kolokoltsov2023variable, STRAKA2018451}. In this case, the limit process is Markov additive, with the generator given by
    \begin{align}
      \mathcal{A}f(x,t) \, = \, &Gf(x,t) \notag \\
    &+b(x)  \partial_t f(x,t) + \int_0^{+\infty} \left( f(x,t+w) -f(x,t) \right) \nu(x,dw)
    \label{59}
    \end{align}
where the operator $(G, \mathcal{D}(G))$ acts on the $x$-variable. In particular, when $G$ has a Courr\`ege-type form and under appropriate technical assumptions on the coefficients, one has that \eqref{59} generates a Markov additive process (see \cite[Section 4]{Meerschaert2014} or \cite[Section 2]{savtoa}). It follows that we can resort to the simplified transition probabilities \eqref{transitionover-NO2} and \eqref{transitionunder-NO4} to determine the Markov processes and the harmonic problems \eqref{1407-a} and \eqref{1407-aa} describing the evolution of $A_{L_t-}$, $A_{L_t}$. In particular, if we assume \eqref{59} with $G$ having a Courr\`ege-type form (as in \cite[Section 3]{savtoa}) we have that the processes $A_t$ and $\sigma_t$ do not jump simultaneously, a.s., and thus $A_{L_t}=A_{L_t-}$, a.s. It follows that we can look, again, at the transition probabilities \eqref{transitionover-NO2} only so
\begin{align}
    p^+_u(s,t;dy,dw) \, = \, \mathds{P}^{(x,0)} \left( A_u \in dy, (t-\sigma_u) \vee 0 \in dw \right).
\end{align}
Therefore the operator $\mathfrak{A}^+$ appearing in \eqref{1407-a} can be determined starting from \eqref{59} just by observing that the second coordinate is reversed and stopped when it crosses zero);  then we have that
\begin{align}
    \mathfrak{A}^+f(x,t) \, = \, &Gf(x,t) -b(x)  \partial_tf(x,t)\notag \\
    &  -\int_0^{+\infty} (f(x,t-w) \mathds{1}_{[w<t]} +f(x,0) \mathds{1}_{[w \geq t]} -f(x,t) ) \nu(x, dw).
    \label{1908}
\end{align}
Setting in \eqref{1908} $b(x)=0$, $\nu(x,dw) =dw \alpha(x) w^{-\alpha(x)-1}/\Gamma(1-\alpha(x))$ one gets the variable order equation studied in \cite{Kolokoltsov2023variable}. Instead, rearranging \eqref{1908} as follows
\begin{align}
    &-\int_0^{+\infty} (f(x,t-w) \mathds{1}_{[w<t]} +f(x,0) \mathds{1}_{[w \geq t]} -f(x,t) ) \nu(x, dw) \notag \\
    = \, & -\frac{\partial}{\partial t} \int_0^t f(x,w) \nu((t-w, +\infty),x) \, dw \, + \,  f(x,0) \, \nu((t, +\infty),x)
\end{align}
one gets the variable order fractional-type equations considered, for example, in \cite{savtoa} (see also \cite{KIAN20181146, kian2018time} for the fractional diffusion case).

\subsection{The undershooted process and coupled equations}
\label{exunder}
In Example \ref{examplecoupled} we provided an easy and explicit (O)CTRW that converges to the process $B_{\sigma_{L_t-}}$ ($B_{\sigma_{L_t}}$). It is clear that the approach can be made more general such that instead of assuming \eqref{jointgausmittag}, we assume that following joint density for $(W_n, J_n)$
\begin{align}
    f_{J,W} (x,s) \, = \, f_s(x) e_\alpha (s) ,
\end{align}
where $f_s(x)$ is the density function of some L\'evy process in $\mathbb{R}^d$ with the symbol $\varphi(\xi)$, $\xi \in \mathbb{R}^d$, and then repeat the same arguments as in \eqref{forosmall} to get the Fourier--Laplace symbol $(\lambda + \varphi(\xi))^\alpha$. This is associated to the process $(M_{\sigma_t}, \sigma_t)_t$ when $\sigma_t$ is a stable subordinator independent from a L\'evy process $M$ having the symbol $\varphi(\xi)$.
The theory of the governing harmonic problem for the overshooted process $M_{\sigma_{L_t}}$ is covered by our general theorems in Section \ref{sectionmain}, even in the case when $M$ is not just a L\'evy but a general Feller process. \red{In the following, we discuss the corresponding harmonic problem for the undershooted process $M_{\sigma_{L_t-}}$ which was already done in \cite{ascione2024} under the same assumptions on $M$ as in Section \ref{sectionmain} and a driftless $\sigma$ satisfying \ref{assphi}}, but we apply our approach from Section \ref{sec:harmonic} to obtain the operator.

In this setting, we are in Markov additive case so we use \eqref{1550} to understand a formal representation of the generator $\mathfrak{A}^-$.
From \eqref{1550}, $\mathfrak{A}^-$ must generate a process as follows. The first term in \eqref{1550}, in this case, corresponds to the subordinate Feller process $M_{\sigma_u-\sigma_0}$ in the first coordinate, and $-\sigma_u$ in the second coordinate, but only if $-\sigma_u$ did not go below level 0 (or in other words if $\sigma_u$ did not go over level $t$). This evolution is then just the subordinated evolution $\overline P_s f(x,t)\coloneqq P_sf(x,t-s) \mathds{1}_{[s < t]}$. On the other hand, the second term in \eqref{1550} corresponds to the event when the process $-\sigma_u$ has crossed the level 0, and in that case we have stopped the second coordinate in 0, while the first coordinate is again the subordinate Feller process $M_{\sigma_u-\sigma_0}$ but stopped in its position before the jump of $-\sigma_u$ below level 0. Since $M_\sigma$ and $\sigma$ have simultaneous jumps, this position is $M_{\sigma_{L_t-}}$. This induces the following representation for $\mathfrak{A}^-$ (under the assumption that $\sigma$ does not have drift as in \cite{ascione2024} and \ref{assphi}), for a suitable function $f$, which is a modification of \eqref{gen-final} that takes into account this stopping:
\begin{align}
    \mathfrak{A}^-f(x,t) \, = \,  \int_0^{+\infty} \left( P_sf(x,t-s) \mathds{1}_{[s < t]} + \underbrace{f(x, 0) \mathds{1}_{[s \geq t]}}_{\text{stopping event}} -f(x,t)\right) \nu(ds).
    \label{genunder}
\end{align}
Rearranging \eqref{genunder} we get
\begin{align}
     \mathfrak{A}^-f(x,t) \, = \,  \int_0^{+\infty} \left( P_sf(x,t-s) \mathds{1}_{[s < t]}  -f(x,t)\right) \nu(ds) + f(x,0) \bar{\nu}(t).
     \label{phiddt-g}
\end{align}

The harmonic problem for the operator having the form \eqref{phiddt-g} has been addressed in \cite[Theorem 5.2]{ascione2024} where it is proved that, under the same assumption of our Theorem \ref{thm:operator}, the unique solution to
\begin{align}
\begin{cases}
    \mathfrak{A}^- q(x,t) = 0, \qquad &(x,t) \in \mathbb{R}^d \times (0, +\infty) \\
    q(x,t) =u(x), & (x,t) \in \mathbb{R}^d \times \{ 0 \},
    \end{cases}
\end{align}
is given by
\begin{align}
    q(x,t) \, = \, \mathds{E}^{(x,0)} u(M_{\sigma_{L_t-}}).
\end{align}

\subsection{\red{A further coupled example: prospects for future work}}
\label{subsec:gaussian-coupling}

We now apply our method to another situation that is relevant in applications, in the context of anomalous diffusion, and is taken from \cite{shlesinger1982random}; it was further developed in Example 5.3 of \cite{becker2004limit}. We stress that, in this case, we provide only the forms of the operators $\mathfrak{A}^{\pm}$. Thus, this example should be regarded as a guideline for applying the harmonic approach when the relevant equation is not available in the literature. Indeed, the well-posedness of the equations associated with the operators given below is an open problem and lies beyond the scope of this example.

We first describe the underlying coupled (O)CTRW, using the notation of Subsection \ref{subsecctrw}. Let $(J_k,W_k)_{k\in\mathbb{N}}$ be i.i.d.\ copies of a random vector $(J,W)$ such that
\[
    \mathds{E}\!\left[e^{-\lambda W}\right]=e^{-\lambda^\beta},
    \qquad
    J\,\big|\,W\sim\mathcal{N}(0,2W^m),
\]
where  $0<\beta<1$ and $m>\beta$. For $c>0$, set
\[
    (J_k^c,W_k^c)
    \coloneqq 
    \left(
        c^{-m/(2\beta)}J_k,\,
        c^{-1/\beta}W_k
    \right).
\]
Thus, $W_k^c$ is a positive $\beta$-stable waiting time satisfying
\[
    \mathds{E}\!\left[e^{-\lambda W_k^c}\right]
    =
    e^{-c^{-1}\lambda^\beta},
    \qquad
    J_k^c\,\big|\,\{W_k^c=s\}
    \sim
    \mathcal{N}(0,2s^m).
\]
In particular, the waiting time and the associated spatial jump are coupled:
longer waiting times produce jumps with larger conditional variance. 

{It is shown in \cite[Example 5.3]{becker2004limit} that the scaling limit converges in distribution to the  L\'evy process  $(A_u,\sigma_u)_{u\geq 0}$, that is, in the notation of Subsection \ref{subsecctrw}, $(S_{[cu]}^c, T_{[cu]}^c) \overset{d}{\longrightarrow}(A_u, \sigma_u)$, as $c \to +\infty$, while from \cite[Theorem 4.1]{Meerschaert2004} follows the convergence in $J_1$ topology as well.

Moreover, in \cite[Example 5.3]{becker2004limit} the authors show that $\sigma_u$ is a $\beta$-stable subordinator, and that the first coordinate $A_u$ admits the following description.}
Define
\begin{equation}
\label{eq:power-jump-subordinator}
    V_u
    \coloneqq 
    \sum_{0<r\leq u}(\Delta\sigma_r)^m.
\end{equation}
The assumption $m>\beta$ ensures that the sum in
\eqref{eq:power-jump-subordinator} is finite almost surely on compact
time intervals. In fact, the process $V$ is a stable subordinator with index
$\beta/m$, constructed from the same jumps as $\sigma$, and its L\'evy measure and Laplace exponent are
\[
    \nu_V(dv)
    =
    \frac{\beta}{m\Gamma(1-\beta)}
    v^{-1-\beta/m}\,dv,
\]
and
\[
    \phi_V(\lambda)
    =
   \frac{\Gamma(1-\beta/m)}{\Gamma(1-\beta)}
        \lambda^{\beta/m},
\]
respectively.
For $(B_u)_{u}$ a standard Brownian motion independent of
$\sigma$, the limiting L\'evy process $(A_u,\sigma_u)_u$ admits the representation
\begin{equation}
\label{eq:coupled-brownian-representation}
    (A_u,\sigma_u)_{u\geq0}
    \overset{{d}}{=}
    \bigl(\sqrt{2}B_{V_u},\sigma_u\bigr)_{u\geq0},
\end{equation}
see \cite{becker2004limit} for additional details.
{Observe that}, conditionally on a jump $\Delta\sigma_u=s$, one has
$\Delta V_u=s^m$ and hence
\[
    \Delta A_u\,\big|\,\{\Delta\sigma_u=s\}
    \sim\mathcal{N}(0,2s^m).
\]
Moreover, $(A_u,\sigma_u)_u$ is a full operator-stable L\'evy process with
exponent
\[
    F=\operatorname{diag}\left(
        \frac{m}{2\beta},\frac{1}{\beta}
    \right),
\]
and satisfies
\[
    (A_{cu},\sigma_{cu})
    \overset{d}{=}
    \left(
        c^{m/(2\beta)}A_u,
        c^{1/\beta}\sigma_u
    \right).
\]
In particular, $(A_u)_u$ is a symmetric stable L\'evy process of index
$2\beta/m$. Although $A_u$ can marginally be viewed as a subordinated
Brownian motion, $(A_u)_u$ and $(\sigma_u)_u$ are not independent, since the
subordinator $V_u$ in \eqref{eq:coupled-brownian-representation} is a
functional of the jumps of $\sigma_u$. The representation \eqref{eq:coupled-brownian-representation} shows
that every jump of size $s$ of $\sigma_u$ produces a simultaneous
Gaussian jump of $A_u$ with variance $2s^m$. Consequently, the joint
L\'evy kernel is
\begin{equation}
\label{eq:gaussian-joint-kernel}
    K(dy,ds)
    =
    k_m(y,s)\,dy\,\nu(ds),
    \qquad
    k_m(y,s)
    \coloneqq 
    \frac{1}{\sqrt{4\pi s^m}}
    \exp\!\left(-\frac{y^2}{4s^m}\right),
\end{equation}
where
\[
    \nu(ds)
    =
    \frac{\beta}{\Gamma(1-\beta)}
    s^{-1-\beta}\,ds.
\]
This representation also shows that the limiting process $(A_u,\sigma_u)_u$ is purely discontinuous and has the form of the generator used in \eqref{genctrwlim}, with only the term involving the kernel $K$ from \eqref{eq:gaussian-joint-kernel}. Moreover, the term $y_i\mathds{1}_{(y,w)\in[-1,1]^{d+1}}\partial_{x_i}f(x,t)$ in \eqref{genctrwlim} integrates to zero because $k_m(\cdot,s)$ is symmetric.

We finally derive the forms of $\mathfrak{A}^\pm$ (note that $(A,\sigma)$ is Markov additive).
For $s>0$, introduce the Gaussian averaging operator
\begin{equation}
\label{eq:gaussian-averaging-operator}
    \mathcal{Q}_s g(x)
    \coloneqq 
    \int_{\mathbb{R}}g(x+y)k_m(y,s)\,dy.
\end{equation}
Thus, $\mathcal{Q}_s$ is the heat semigroup evaluated at time $s^m$;
in particular, the family $(\mathcal{Q}_s)_{s \geq 0}$ is not a semigroup
unless $m=1$. The symmetry of the kernel makes the truncation term in
Remark \ref{rem25} vanish after integration with respect to $y$. Moreover, the
condition $m>\beta$ guarantees the local integrability of the
compensation-free expressions below.

For the process $A_{L_t}$, formula \eqref{gen:1026} amounts to
reversing the second coordinate and stopping it at zero, while retaining
the spatial component of the jump that crosses zero. It follows that
\begin{align}
\mathfrak{A}^{+}f(x,t)
    &=
    \int_0^{+\infty}
    \left(
        \mathcal{Q}_s
        f\bigl(\cdot,(t-s)\vee0\bigr)(x)
        -f(x,t)
    \right)\nu(ds)
    \label{eq:Aplus-gaussian}
    \\
    &=
    \int_0^t
    \left(
        \mathcal{Q}_s f(\cdot,t-s)(x)-f(x,t)
    \right)\nu(ds)
    \notag\\
    &\quad+
    \int_t^{+\infty}
    \left(
        \mathcal{Q}_s f(\cdot,0)(x)-f(x,t)
    \right)\nu(ds).
    \notag
\end{align}

For the process $A_{L_t-}$, we instead use the transition
probabilities \eqref{1550}, following the argument leading to \eqref{genunder}. If $s<t$,
the coupled jump is performed. If $s\geq t$, the second coordinate is
stopped at zero and the spatial component of the crossing jump is not
included. Therefore,
\begin{align}
\mathfrak{A}^{-}f(x,t)
    &=
    \int_0^{+\infty}
    \Bigl(
        \mathcal{Q}_s f(\cdot,t-s)(x)\mathbf{1}_{[s<t]}
        +f(x,0)\mathbf{1}_{[s\geq t]}
        -f(x,t)
    \Bigr)\nu(ds)
    \label{eq:Aminus-gaussian}
    \\
    &=
    \int_0^t
    \left(
        \mathcal{Q}_s f(\cdot,t-s)(x)-f(x,t)
    \right)\nu(ds)
    +
    \bigl(f(x,0)-f(x,t)\bigr)\bar{\nu}(t).
    \notag
\end{align}
Hence, the two operators have the same contribution from jumps satisfying $s<t$ and differ only in their treatment of the jump that crosses the temporal boundary:
\begin{equation}
\label{eq:Aplus-Aminus-gaussian}
    \bigl(\mathfrak{A}^{+}-\mathfrak{A}^{-}\bigr)f(x,t)
    =
    \int_t^{+\infty}
    \left(
        \mathcal{Q}_s f(\cdot,0)(x)-f(x,0)
    \right)\nu(ds).
\end{equation}
Accordingly, $\mathfrak{A}^{+}$ is associated with the overshooting CTRW limit $A_{L_t}$, whereas $\mathfrak{A}^{-}$ is associated with the CTRW limit $(A_{L_t-})^{+}$. In the present case, $(A_{L_t-})^{+}=A_{L_t-}$ almost surely. Indeed, this follows from Proposition \ref{p:1333}, since $A_u$ and $\sigma_u$ have simultaneous jumps and since the $\beta$-stable subordinator $\sigma_u$ does not hit points. The well-posedness of the harmonic problems associated with $\mathfrak{A}^{\pm}$ must be addressed separately.

\subsection{\red{Non-Markov additive example}}
\label{subsec:non-markov-additive-example}
We conclude with a simple example in which the limiting process $(A_u,\sigma_u)_{u\geq0}$ is not Markov additive. Consequently, the construction of Subsection \ref{ss:general} leads to a genuine family of operators $\{{}^{t}\mathfrak{A}^{\pm}\}_{t\geq0}$, and the parameter $t$ cannot be removed as in Subsection \ref{ss:additive}. The example is deliberately artificial rather than motivated by a specific application. Its purpose is to illustrate how our methodology applies to this context and the difficulties that arise in establishing the well-posedness of the associated harmonic problems: the explicit dependence of the operators on $t$ prevents the direct use of Laplace-transform techniques, which are otherwise standard tools for non-local equations of this type.

It is not hard to construct a Markov chain $(S_{[cu]}^c,T_{[cu]}^c)_{u}$ whose generator $\mathcal{A}^c$ converges to
\begin{equation}
\label{eq:non-ma-limit-generator}
    \mathcal{A}_c f(x,s)   \longrightarrow   \mathcal{A}f(x,s)  \coloneqq   \partial_s f(x,s)  + \lambda(s)
    \bigl(
        f(x+1,s+1)-f(x,s)
    \bigr),
\end{equation}
as $c\to+\infty$. The limit process $(A,\sigma)$ generated by $\mathcal{A}$ has the jump kernel
\[
    K(x,s;dy,dw)
    =
    \lambda(s)\delta_1(dy)\delta_1(dw).
\]
Hence, starting from $(A_0,\sigma_0)=(x,s)$, one has $A_u=x+C_u$ and $\sigma_u=s+u+C_u$.
Here, $C$ is a jump-by-one process with predictable intensity
$\lambda(\sigma_{u-})$, and every jump satisfies $\Delta A_u=\Delta\sigma_u=1.$
The deterministic drift in the second coordinate makes $\sigma$
strictly increasing and unbounded.
Since $\lambda$ is non-constant, the jump kernel depends explicitly on
the second coordinate $s$. Hence $(A,\sigma)$ is not Markov additive.
Probabilistically, two copies starting from $(x,s_1)$ and $(x,s_2)$
have different instantaneous jump rates whenever
$\lambda(s_1)\neq\lambda(s_2)$, even though their first coordinates
coincide.

We now determine the operators associated with the CTRW and OCTRW
limits. Fix a level $t>0$. In the notation of
Subsection \ref{ss:general}, if the starting point of the processes $\,^tZ^\pm$ has the second coordinate $v$, the original process starts from $s=t-v$, so $v$ is the residual time, and the instantaneous jump rate is therefore $\lambda(t-v)$.

Before the boundary is reached, the deterministic part of $\sigma$
makes the residual coordinate decrease with constant speed 1. A simultaneous
jump of $(A_u,\sigma_u)$ sends
\[
    (x,v)\longmapsto(x+1,v-1)
\]
whenever $v>1$. If $0<v\leq1$, the jump reaches or crosses the temporal
boundary. In the OCTRWs limit case, i.e. for $\,^tZ^+$, the spatial jump is retained, while
in the CTRWs limit case, i.e. for $\,^tZ^-$, the jump is discarded.

It follows that, for $v>0$ and for functions $f$ that are continuously
differentiable in the second variable,
\begin{equation}
\label{eq:non-ma-A-plus}
    {}^{t}\mathfrak{A}^{+}f(x,v)=-\partial_v f(x,v) +  \lambda(t-v)\left[f\bigl(x+1,(v-1)\vee0\bigr)-f(x,v)\right].
\end{equation}
Indeed, when $0<v\leq1$, the OCTRW includes the spatial component of the
jump crossing the boundary, and the process is stopped at $(x+1,0)$.

For the undershooting process, the crossing jump is not included.
Therefore,
\begin{align}
\label{eq:non-ma-A-minus}
    {}^{t}\mathfrak{A}^{-}f(x,v) &= -\partial_v f(x,v) + \lambda(t-v)\Big[\mathds{1}_{\{v>1\}}f(x+1,v-1) + \mathds{1}_{\{0<v\leq1\}}f(x,0) - f(x,v) \Big].
\end{align}
Both transformed processes are absorbed at the temporal boundary, so
that ${}^{t}\mathfrak{A}^{\pm}f(x,0)=0.$
The operators agree away from the crossing region. More precisely,
\begin{equation}
\label{eq:non-ma-operator-difference}
    \bigl(
        {}^{t}\mathfrak{A}^{+}
        -
        {}^{t}\mathfrak{A}^{-}
    \bigr)f(x,v)
    =
    \lambda(t-v)
    \mathds{1}_{\{0<v\leq1\}}
    \bigl(
        f(x+1,0)-f(x,0)
    \bigr).
\end{equation}

Accordingly, the general family of harmonic problems in \eqref{1407} takes
the form
\begin{equation}
\label{eq:non-ma-harmonic-problems}
    \begin{cases}
        {}^{t}\mathfrak{A}^{\pm}q_t^{\pm}(x,v)=0,
        &
        (x,v)\in\mathbb{R}\times(0,+\infty),
        \\[1mm]
        q_t^{\pm}(x,0)=u(x),
        &
        x\in\mathbb{R}.
    \end{cases}
\end{equation}
By the construction of Subsection \ref{ss:general}, the values on the diagonal are
expected to satisfy
\[
    q_t^{+}(x,t)
    =
    \mathds{E}^{(x,0)}
    \bigl[u(A_{L_t})\bigr],
    \qquad
    q_t^{-}(x,t)
    =
    \mathds{E}^{(x,0)}
    \bigl[u(A_{L_t-})\bigr].
\]
The explicit dependence on $t$ in
\eqref{eq:non-ma-A-plus} and \eqref{eq:non-ma-A-minus} is not merely
notational. In the Markov additive setting, the parameter $t$ can be
removed and one obtains a single operator $\mathfrak{A}^{\pm}$ acting
on the physical time variable. This temporal homogeneity is what makes
Laplace-transform techniques effective. In particular, in Section~\ref{sectionmain}
the operator $\mathfrak{A}^{+}$ governing the overshooted process is
independent of the terminal level, and the proof of Theorem~\ref{thm:operator} is
based on taking the Laplace transform of the equation in the time
variable.

Here, instead, the unknown is the entire family
$\{q_t^{\pm}(x,v)\}_{t\geq0}$, and the operators contain the mixed
coefficient $\lambda(t-v)$. For example, taking the Laplace transform
with respect to $t$ would produce terms of the form
\[
    \int_0^{+\infty}
    e^{-\theta t}
    \lambda(t-v)q_t^{\pm}(x,v)\,dt,
\]
which are not scalar multiples of the Laplace transform of
$t\mapsto q_t^{\pm}(x,v)$. Thus the transform does not diagonalize the
time-dependent operator and does not yield the closed algebraic
resolvent identity used in the overshooting problem of Section \ref{sectionmain} and typical of the techniques used to deal with the equation arising in other known examples in this section, where the Markov additive property is available.
Moreover, the quantity describing the (O)CTRW limit is obtained only
after the diagonal evaluation $v=t$. For this reason, the
one-dimensional Laplace-transform method used in the time-homogeneous
overshooted case (and in other cases related to Markov addivity) is not well adapted to the present example; one must
instead work directly with the family of time-dependent harmonic
problems \eqref{eq:non-ma-harmonic-problems}, or with the corresponding
time-inhomogeneous backward equations.




\appendix

\section{Proofs of auxiliary results}
\label{auxiliary}
\subsection{On the core of (subordinated) Feller processes}
\begin{lemma}\label{app:core}
Let $(M_u)_{u\ge0}$ be a Feller process on a locally compact separable metric space $E$, with a generator $(G,\DD(G))$, and $(M^\phi_u)_{u\ge0}$ a subordinated Feller process $(M_u)_{u\ge0}$ by a subordinator $(\sigma_u)_u$ with the Laplace exponent $\phi$, and $(G^\phi,\DD(G^\phi))$ its generator.
    \begin{enumerate}[label=(\alph*)]  
        \item\label{a}If $\mathcal{C}$ is an operator core for $(G,\DD(G))$, then it is also for $(G^\phi,\DD(G^\phi))$.
        \item \label{b}
        If $(\Gamma_u)_u$ is a simple translation in $\R$, then $(M_u,\Gamma_u)_{u\ge0}$ is a Feller process in $E\times \R$. In the case $E=\R^d$, if $C_c^\infty(\R^d)$ is an operator core for $(G,\DD(G))$, then the space $C_c^\infty(\R^{d})\otimes  C_c^\infty(\R)\coloneqq \{(x,t)\mapsto f(x)g(t):f\in C_c^\infty(\R^{d}),\, g\in C_c^\infty(\R)\} $ is an operator core for the generator of $(M_u,\Gamma_u)_{u\ge0}$.
    \end{enumerate}
\end{lemma}
\begin{proof}
        First we prove part \ref{a}. Let $\phi(\lambda)=b\lambda+\int_0^{+\infty} (1-e^{-\lambda s})\nu(ds)$. By Phillips' theorem \cite[Theorem 15.6]{bernstein}, the  generator $(G^\phi,\DD(G^\phi))$ has a core $\DD(G)$ and it holds
        \begin{align}\label{1138}
            G^\phi f=b\,Gf+\int_0^{+\infty} (P_sf-f)\nu(ds),\quad f\in \DD(G),
        \end{align}
        where $(P_t)_{t\ge0}$ is the semigroup of $(M_u)_{u\ge0}$.
        
        Take now $\varepsilon>0$ and $f\in \DD(G^\phi)$. By the definition of the core, there exists $f_\varepsilon\in \DD(G)$ such that $\|f-f_\varepsilon\|_\infty+\|G^\phi f-G^\phi f_\varepsilon\|_\infty<\varepsilon$. However, $\mathcal{C}$ is a core for $\DD(G)$, so there exists $\overline f_\varepsilon\in \mathcal{C}$ such that $\|f_\varepsilon-\overline f_\varepsilon\|_\infty+\|G f_\varepsilon-G \overline f_\varepsilon\|_\infty<\varepsilon$. Hence,
        \begin{align}
            \|f-\overline f_\varepsilon\|_\infty+\|G^\phi f-G^\phi \overline f_\varepsilon\|_\infty&\le \|f-f_\varepsilon\|_\infty+\|G^\phi f-G^\phi f_\varepsilon\|_\infty\\
            &\quad+\|f_\varepsilon-\overline f_\varepsilon\|_\infty+\|G^\phi f_\varepsilon-G^\phi \overline f_\varepsilon\|_\infty\\
            &\le 2\varepsilon+\|G^\phi f_\varepsilon-G^\phi \overline f_\varepsilon\|_\infty.\label{1445-a}
        \end{align}
        However, by applying \cite[Eq. (15.3)]{bernstein} and \eqref{1138}, we get
        \begin{align}
            \|G^\phi f_\varepsilon-G^\phi \overline f_\varepsilon\|_\infty&\le b\|G f_\varepsilon-G \overline f_\varepsilon\|_\infty\\
            &\qquad+\int_0^\infty\min\{s \|G f_\varepsilon-G \overline f_\varepsilon\|_\infty,\| f_\varepsilon- \overline f_\varepsilon\|_\infty\}\nu(ds)\\
            &\le C(\phi)\varepsilon.\label{1445-b}
        \end{align}
        By collecting \eqref{1445-a} and \eqref{1445-b}, we get that $\mathcal{C}$ is a core for $(G^\phi,\DD(G^\phi))$.

        Now we move to the proof of the part \ref{b}. The semigroup of $(M_u,\Gamma_u)_{u\ge0}$ is given by
        \begin{align}
            \wt P_uf(x,t)\coloneqq P_uf(\cdot,t+u)(x)=\ex^x(f(M_u,t+u)),
            \label{Ptilda}
        \end{align}
        and it is easy to show by definition that it maps $C_0(E\times\R)$ to $C_0(E\times\R)$, and that it is strongly continuous at $u=0$, i.e. that it is a Feller semigroup.

        The core part, under the additional assumption $E=\R^d$, comes essentially from \cite[Theorem 4.1]{hernandez2017}. Denote by $(G+\partial_t)$ the infinitesimal generator of $(\wt P_u)_u$, and assume that $C_c^\infty(\R^d)$ is a core for $(G,\DD(G))$. Since $\DD(\partial_t)=C^1_0(\R^d)$, the space of differentiable functions vanishing at infinity (as well as their derivative), we have that $C_c^\infty(\R^d){\otimes} C_c^\infty(\R)\subset \textrm{span}\{(x,t)\mapsto f(x)g(t):f\in \DD(G),g\in \DD(\partial_t)\}\subset \DD(G+\partial_t)$. Since $C_c^\infty(\R^d)\otimes  C_c^\infty(\R)$ forms a subalgebra, by Stone-Weierstrass theorem \cite[Theorem A.0.5]{applebaum-semigroups}, both $C_c^\infty(\R^d)\otimes  C_c^\infty(\R)$ and $\textrm{span}\{(x,t)\mapsto f(x)g(t):f\in \DD(G),g\in \DD(\partial_t)\}$ are dense in $C_0(\R^{d+1})$ and hence in $\DD(G+\partial_t)$. Moreover, $\textrm{span}\{(x,t)\mapsto f(x)g(t):f\in \DD(G),g\in \DD(\partial_t)\}$ is invariant under $\wt P_u$ so \cite[Theorem 1.3.18]{applebaum-semigroups} implies that $\textrm{span}\{(x,t)\mapsto f(x)g(t):f\in \DD(G),g\in \DD(\partial_t)\}$ is a core for $\DD(G+\partial_t)$. Furthermore, since $C_c^\infty(\R^d)$ is a core for $(G,\DD(G))$ and $C_c^\infty(\R)$ is a core for $(\partial_t,\DD(\partial_t))$, by similar approximation technique as in part \ref{a}, we get that $C_c^\infty(\R^d)\otimes  C_c^\infty(\R)$ is a core for $\DD(G+\partial_t)$.
\end{proof}
\subsection{Proof of Lemma \ref{l:aux-G-q(t)}}
Proof of \ref{l-item-1}: Recall that we can rewrite $q(x,t)$ as 
\begin{align}
q(x,t)=\int_{[t,\infty)} P_su(x) \, \pr^{(x,0)}(D_t\in ds), 
\end{align}
and so
	\begin{align}\label{eq:Gq-exists}
		\frac{P_hq(x,t)-q(x,t)}{h}=\int_{\red{[t,{+\infty})}} \frac{1}{h}\left(P_{s+h}u(x)-P_{s}u(x)\right)\pr^{(x,0)}(D_t\in ds),
	\end{align}
    where we used \cite[Proposition 1.1.7]{abhn} to move $P_h$ inside the integral.
	Note that since $u\in \DD(G)$, we have that $Gu\in C_0(E)$ and $P_sGu=GP_su$. Furthermore, we have that 
    \begin{align}
        \left(P_{s+h}u(x)-P_{s}u(x)\right)/h\to P_sGu(x), 
    \end{align}
    uniformly in $x$, and by the well-known equality
    \begin{align}
        P_{s+h}u-P_su \, = \, \int_0^h GP_{s+w}u \, dw
    \end{align}
     we can see that the term inside the integral in \eqref{eq:Gq-exists} is bounded, for $u \in \mathcal{D}(G)$, by $\|P_sGu\|\le \|Gu\|$. Since this bound is uniform in $h$, by the dominated convergence theorem, we obtain
	\begin{align}
		Gq(x,t)=\int_{\red{[t,{+\infty})}} P_sGu(x)\pr^{(x,0)}(D_t\in ds)=\ex^{(x,0)}[Gu(Y_t)].
        \label{a11}
	\end{align}
    The claimed inequality on the norms follows directly from \eqref{a11}.

    \red{We are left to prove vanishing at infinity in $x$, locally uniformly in $t>0$. Take $T>0$ and $\varepsilon>0$. Choose $T^*>T$ so that $\overline \nu(T^*-T)\int_0^Tu^\phi(h)dh<\varepsilon/2$. This is possible since $\overline \nu (r)\to 0$ as $r\to +\infty$ and since by \ref{assphi} the potential density $u^\phi$ is locally integrable, see \cite[Theorem 13.3]{bernstein}. We have, for all $t\in[0,T]$ and $x\in E$, by \eqref{eq:density-D}
    \begin{align}
        \pr^{(x,0)}(D_t>T^*)&=\int_0^t\int_{T^*}^{+\infty}\nu(ds-h)u^\phi(h)dh=\int_0^t\overline \nu(T^*-h)u^\phi(h)dh\\
        &\le\overline \nu(T^*-T)\int_0^Tu^\phi(h)dh<\varepsilon/2.
    \end{align}
    \noindent
    Further, drawing inspiration from the argument in the proof of \cite[Proposition~5.2]{ascione2024}, since the semigroup $(P_t)_t$ is strongly continuous and Feller, for every \(t\in[0,T^*]\) there exist a compact set \(K_t\subset E\) and \(\delta_t>0\) such that
        \[
            |P_su(x)|<\frac{\varepsilon}{2}, \quad s\in\mathcal{I}_t\coloneqq[(t-\delta_t)\vee0,t+\delta_t],\, x\in E\setminus K_t.
        \]
        Choosing a finite subcover \(\{\mathcal{I}_{t_i}\}_{i=1}^n\) of \([0,T^*]\) and setting \(K\coloneqq\bigcup_{i=1}^nK_{t_i}\), we obtain 
        \[
            |P_su(x)|<\frac{\varepsilon}{2},  \quad s\in[0,T^*],\, x\in E\setminus K.
        \]
    Therefore, for all $(x,t)\in \big(E\setminus K\big)\times [0,T]$, we have
    \begin{align}
        |q(x,t)|\le \int_0^{T_*}|P_su(x)|\pr^{(x,0)}(D_t\in ds)+\|u\|_{\infty} \pr^{(x,0)}(D_t>T^*)<\varepsilon,
    \end{align}
    i.e., $q$ vanishes at infinity in $x$, locally uniformly in $t>0$.
    }
	
	Proof of \ref{l-item-2}: \red{using the distribution of $D_t$ given in \eqref{eq:density-D}}, we have that (see also \eqref{eq:Laplace-q})
   \begin{align}
   \begin{split}
       \wt q(x,\lambda)=\int_0^{+\infty} \int_0^{+\infty} &e^{-\lambda w}P_{s+w}u(x) \left(\frac{1-e^{-\lambda s}}{\lambda}\right)\nu(ds)u^\phi(w)dw\\&\red{+b\int_0^{+\infty} e^{-\lambda t} P_tu(x)u^\phi(t)dt}.
   \end{split}
    \end{align}
   In the same spirit of Item \ref{l-item-1}, we obtain
   \begin{align}
   \begin{split}
       \frac{P_h \widetilde q(x,\lambda) - \widetilde q(x,\lambda)}{h} \, = \, \int_0^{+\infty} \int_0^{+\infty} &e^{-\lambda w}\frac{P_h-\mathds{I}}{h}P_{s+w}u(x) \left(\frac{1-e^{-\lambda s}}{\lambda}\right)\nu(ds)u^\phi(w)dw\\
       & \red{+b\int_0^{+\infty} e^{-\lambda t} \frac{P_h-\mathds{I}}{h}P_tu(x)u^\phi(t)dt}
   \end{split}
   \end{align}
   and the function inside the first integral is bounded uniformly in $h$, \red{with the bound}
   \begin{align}
   \lambda^{-1} \left( 1-e^{-\lambda s} \right) e^{-\lambda w}   \left\| Gu \right\|_\infty
   \end{align}
   \red{which is integrable against $\nu(ds) u^\phi(w)dw$, while the function in the second integral is bounded by $e^{-\lambda t}\|Gu\|_\infty u^\phi(t)$. Hence, by the dominated convergence we get}   
   $\wt q(\cdot,\lambda)\in \DD(G)$ and $G \wt q(x,\lambda)=\wt{Gq(x,\cdot)}(\lambda)$.

	Proof of \ref{l-item-3}.  The proof can be done in a similar way as in \cite[Proposition 5.1]{ascione2024} so we repeat only the crucial steps and emphasize the differences. Recall the semigroup \eqref{Ptilda} of the process $(M_u, \Gamma_u)_{u \geq 0}$, i.e., $(\wt P_t,t\ge0)$ on $C_0 (E\times \R)$ and $(G + \partial_t)$ be the corresponding infinitesimal generator (see Lemma \ref{app:core}, Item \ref{b}). Consider the sequence $\eta_N\in C_c^\infty(\R)$ such that $0\le \eta_N\le 1$, $\eta_N\equiv1$ in $[-N,N]$, $\eta_N\equiv 0$ in $(-N-1,N+1)^c$, and such that $\|\eta_N\|_{C^2(\R)}=C_\eta$ is independent of $N\in \N$; then $u_N(x,t)=u(x)\eta_N(t)\in \DD(G + \partial_t)$ and 
    \begin{align}\label{1458}
    (G+\partial_t)u_N(x,t)=Gu(x)\eta_N(t)+u(x)\partial_t\eta_N(t).
    \end{align}
	Next, recall by \eqref{1640} the generator $\mathcal{A}$ of the corresponding subordinate semigroup in the sense of Bochner, with the Bernstein function $\phi$. Obviously, $L_t$ is the stopping time for the natural filtration generated by the corresponding process $(M_{\sigma^0},\sigma)$, since it is the hitting time of $(t,{+\infty})$ for $\sigma$. 
	
	Define
	\begin{align}
		q_N(t,x,s)\coloneqq \ex^{(x,s)}[u_N(\red{M_{D_t}},D_t)]=\ex^{(x,s)}[u_N(M_{\sigma^0_{L_t}},\sigma_{L_t})],
	\end{align}
	which by Dynkin's formula, see \cite[Eq. (1.55)]{LevyMatters3}, satisfies
	\begin{align}
		q_N(t,x,s)=u_N(x,s)+\ex^{(x,s)}\left[\int_0^{L_t} \mathcal{A} u_N(M_{\sigma^0_h},\sigma_h)dh\right].
	\end{align}
    Note that, by  \eqref{1458} and the calculations in \cite[p. 12]{ascione2024}, we have 
    \begin{align}
        \|\mathcal{A} u_N\|_\infty\le C_\phi(\|Gu\|_\infty+\|u\|_\infty).
    \end{align}
	Fix now $T>1$ and for $h\in (-1,1)\setminus\{0\}$ and $\varphi:(-1,T+1)$ define
	\begin{align}
		D^h\varphi(t)=\frac{\varphi(t+h)-\varphi(t)}{h}.
	\end{align}
	Then it holds by the same calculations as in \cite{ascione2024} (see, in particular, \cite[p. 12]{ascione2024}) that
	\begin{align}
		|D^hq_N(t,x,0)|&\le C_\phi\big(\|Gu\|_\infty+\|u\|_\infty\big) \times
        \begin{cases}
            u^\phi(t),&h>0,\\
            \frac{U(t)-U(t+h)}{|h|}
            ,&h<0,\,t+h>0,\\
            \frac{U(t)}{|h|}, &h<0,\, t+h\le 0,
        \end{cases}\label{eq:1833}\\
        &\le 2C_\phi\big(\|Gu\|_\infty+\|u\|_\infty\big)\left(u^\phi(t)+u^\phi(t/2)+\frac{U(t)}{t}\right).
	\end{align}
	By taking $N\to {+\infty}$ we obtain the same bounds for $D^hq(t,x)$, using dominated \red{convergence} theorem.
	Take now an arbitrary sequence $h_n\to 0$. Denote by $g_n(t)=D^{h_n}q(t,x)$. We proceed by showing that $g_n(t)$ has a subsequence that converges weakly in $L^1(0,T)$ (i.e. with respect to all $L^\infty(0,T)$ functions), and that the limit is $\partial_t g(t,x)$. This can be done in the very same way as \cite[pp. 12-13]{ascione2024} proving that the sequence $\{g_n:n\in \N\}$ is uniformly integrable and uniformly bounded in $L^1(0,T)$. It is in the previous step that we use the assumption on the log-integrability of $u^\phi$. It follows by the Dunford-Pettis theorem \cite[Theorem 4.30]{Brezis-SS+PDE} that there is a function $g\in L^1(0,T)$ and a subsequence of $(h_n)_n$, \red{which} we denote again $(h_n)_n$, such that $g_n\rightharpoonup g$ in the weak topology of $L^1(0,T)$, i.e. 
	\begin{align}\label{eq:weak-conv-q-1}
		\lim_{n\to{+\infty}}\int_0^T g_n(t) \varphi(t)dt= \int_0^T g(t) \varphi(t)dt,\quad \varphi\in L^\infty(0,T).
	\end{align}
	In particular, for $\varphi\in C_c^\infty(0,T)$ we get
	\begin{align}
		\int_0^Tg_n(t) \varphi(t)dt&=\int_0^{T+1} \frac{q(x,t+h_n)-q(x,t)}{h_n} \varphi(t)dt\nonumber\\
		&=-\int_0^{T+1} q(x,t)D^{h_n}\varphi(t)dt\to -\int_0^{T+1} g(t)\partial_t\varphi(t)dt\nonumber\\
		&=-\int_0^{T} g(t)\partial_t\varphi(t)dt,\label{eq:weak-conv-q-2}
	\end{align}
	as $n\to{+\infty}$, where the convergence is obtained by the dominated convergence theorem. By comparing \eqref{eq:weak-conv-q-1} and \eqref{eq:weak-conv-q-2}, we see that $g=\partial_t q(x,t)$ in the distributional sense. In particular, $q$ is in the Sobolev space $W^{1,1}(0,T)$, hence there is an absolutely continuous version of $q$ with the a.e. derivative $g$, see \cite[Chapter 5]{evans_pde}. However, $t\mapsto q(x,t)$ is continuous by the dominated convergence theorem since $u\in C_0(E)$, so \red{$q$ coincides with its continuous version}. Hence, $g(t)=\partial_t q(t,x)$ for a.e. $t\in (0,T)$. Since $T$ was arbitrary, we can repeat this procedure for every $T>1$ to get the existence of $\partial_tq(t,x)$ in $(0,{+\infty})$.
		To finish the proof, note that we are now allowed to take the limit in \eqref{eq:1833} with $h\to 0$ (and considering $q$ instead of $q_N$) to get the bound $|\partial_t q(t,x)|\le C_\phi\big(\|Gu\|_\infty+\|u\|_\infty\big) u^\phi(t)$ for a.e. $t>0$.

\bibliographystyle{abbrv}
\bibliography{References}

\bigskip

\noindent{\bf Ivan Bio\v{c}i\'c}

\noindent Department of Mathematics, Faculty of Science, University of Zagreb, Zagreb, Croatia,

\noindent Department of Mathematics “Giuseppe Peano”, University of Turin, Turin, Italy,

\noindent Email: \texttt{ivan.biocic@unito.it}, \texttt{ivan.biocic@math.hr}

\bigskip
\noindent {\bf Bruno Toaldo}

\noindent Department of Mathematics “Giuseppe Peano”, University of Turin, Turin, Italy,

\noindent Email: \texttt{bruno.toaldo@unito.it}

\end{document}